\documentclass[12pt]{amsart}
\usepackage[top=4.2cm, bottom=4cm, left=2.4cm, right=2.4cm]{geometry}
\usepackage[utf8]{inputenc}
\usepackage[USenglish]{babel}
\usepackage[T1]{fontenc} 
\usepackage{mathrsfs}
\usepackage{mathtools}
\usepackage{amsmath}
\usepackage{amssymb,epsfig}
\usepackage[width=0.9\textwidth]{caption}
\usepackage{subcaption}
\usepackage{comment}
\usepackage{soul}
\usepackage[all]{xy}
\usepackage{enumitem}

\usepackage{amsthm}
\usepackage[absolute]{textpos}
\usepackage{mathtools,emptypage}

\usepackage[bookmarks=true]{hyperref}
\usepackage{xcolor}
\hypersetup{
	colorlinks,
	linkcolor={red!50!black},
	citecolor={blue!50!black},
	urlcolor={blue!80!black},
}

\usepackage{tipa} 
\usepackage{tikz}

\mathtoolsset{showonlyrefs}

\theoremstyle{definition}
\newtheorem{defin}{Definition}[section]
\newtheorem{ex}[defin]{Example}
\theoremstyle{plain}
\newtheorem{theo}[defin]{Theorem}
\newtheorem{lemma}[defin]{Lemma}
\newtheorem{obs}[defin]{Remark}
\newtheorem{prop}[defin]{Proposition}
\newtheorem{cor}[defin]{Corollary}

\newtheorem{question}{Question}
\newtheorem*{question-no-number}{Question}

\newtheorem*{theorem-no-number}{Theorem}
\newtheorem{theorem}{Theorem}
\newtheorem{corollary}[theorem]{Corollary}

\newtheorem*{T1}{Theorem~\ref{theo:intro_qi_implies_biLip}}
\newtheorem*{T2}{Theorem~\ref{theo:intro_qi_cocompact}}
\newtheorem*{T3}{Theorem~\ref{theo:intro_negative_answer}}
\newtheorem*{T4}{Corollary~\ref{cor:intro_hom_mfd_not_open}}
\newtheorem*{T5}{Theorem~\ref{theo:intro_sphere_degree}}
\newtheorem*{T6}{Theorem~\ref{theo:intro_2_volume}}
\newtheorem*{T7}{Theorem~\ref{theo:intro_2}}
\newtheorem*{T8}{Theorem~\ref{theo:intro_homology_qi_R^n}}

\newcommand{\restr}[1]{\lower3pt\hbox{$|_{#1}$}}

\newcommand{\N}{\mathbb{N}}

\newcommand{\R}{\mathbb{R}}
\newcommand{\Z}{\mathbb{Z}}
\newcommand{\sfd}{{\sf d}}
\newcommand{\X}{{\rm X}}
\newcommand{\Y}{{\rm Y}}
\newcommand{\CAT}{\textup{CAT}}
\newcommand{\Isom}{\textup{Isom}}

\newcommand{\Cov}{\textup{Cov}}

\renewcommand{\d}{{\mathrm d}}

\renewenvironment{abstract}
{\par\noindent\textbf{\abstractname.}\ \ignorespaces}
{\par\medskip}

\title{CAT$(0)$ spaces quasi-isometric to Euclidean spaces}

\author[Nicola Cavallucci]{Nicola Cavallucci}
\address[Nicola Cavallucci]{Département de mathématiques,
Université de Fribourg,
Ch. du musée 23
CH-1700 Fribourg,
Switzerland}
\email{n.cavallucci23@gmail.com}
\author[Andrea Sambusetti]{Andrea Sambusetti}
\address[Andrea Sambusetti]{Dipartimento di Matematica G.Castelnuovo, "La Sapienza" Università di Roma, Piazzale Aldo Moro, 00185 Roma, Italy}
\email{sambuset@mat.uniroma1.it}


\begin{document}
	\maketitle
	\begin{abstract}
		\footnotesize
		We show that if a proper, geodesically complete, $\CAT(0)$ homology manifold is quasi-isometric to the Euclidean space $\R^n$ then it is homeomorphic to $\R^n$. 
        On the other hand, we show that there exist proper, geodesically complete, $\CAT(0)$ spaces quasi-isometric to $\R^n$, which are not homeomorphic to it. We prove that our example is sharp in a suitable sense.
        Finally, we provide an example of a sequence of proper, geodesically complete, $\CAT(0)$ spaces that are not homology manifolds and that converge in the Gromov-Hausdorff sense to a topological manifold: this shows that the set of topological manifolds is not open in the class of proper, geodesically complete, $\CAT(0)$ spaces.
	\end{abstract}

	\tableofcontents
    \newpage

\section{Introduction}

Given $L\ge 1$ and $C\ge 0$, two metric spaces $(\X,\sfd_\X), (\Y,\sfd_\Y)$ are $(L,C)$-quasi-isometric if there exists a not necessarily continuous function $f\colon \X \to \Y$ such that:
\begin{itemize}
	\item[(i)] $\frac{1}{L} \sfd_\X(x,x') - C \le \sfd_\Y(f(x),f(x')) \le L \sfd_\X(x,x') + C$ for every $x,x' \in \X$;
	\item[(ii)] for every $y\in \Y$ there exists $x \in \X$ such that $\sfd_\Y(y,f(x)) \le C$.
\end{itemize}
Two metric spaces are {\em quasi-isometric} if they are $(L,C)$-quasi-isometric for some $L\ge 1, C\ge 0$. 
They are {\em  $L$-biLipschitz homeomorphic} if the same holds for $C=0$.
\vspace{1mm}

In this paper we address the following:

\begin{question}
\label{question:intro}
    If $(\X,\sfd)$ is a proper, geodesically complete, $\CAT(0)$ space which is $(L, C)$-quasi-isometric to $\R^n$, then is it homeomorphic to $\R^n$?
\end{question} 

We notice that the constant $C$ in the definition of quasi-isometry plays no role in Question \ref{question:intro}.\linebreak Indeed, for every proper, geodesically complete, $\CAT(0)$ space $(\X,\sfd)$ which is $(L,C)$-quasi-isometric to $\R^n$ and for every $c>0$, a suitable rescaling $(\X,\lambda\sfd)$ is $(L,c)$-quasi-isometric to $\R^n$.
On the other hand, we will see that the value of the constant $L$ matters.
\vspace{2mm}

A first, positive answer to Question \ref{question:intro} under the additional assumption that  $(\X,\sfd)$ has a discrete, cocompact group of isometries is essentially known  because of \cite[Corollary C]{AB97}. In fact, applying an unpublished result in \cite{Bosche2011}, the following   result holds as soon as a proper, geodesically complete, $\CAT(0)$ space has the full isometry group which is cocompact.


\begin{theorem}
	\label{theo:intro_qi_cocompact}
	Let $(\X,\sfd)$ be a proper, geodesically complete, $\CAT(0)$ space such that $\Isom(\X,\sfd)$ acts cocompactly on $\X$. If $(\X,\sfd)$ is quasi-isometric to $\R^n$, then $(\X,\sfd)$ is isometric to $\R^n$.
 \end{theorem}
 
 In Section  \ref{sec:Theorems_Nagano_and_cocompact} we will provide a short,  self-contained proof of this fact. 
 
One of the first results of this paper answers Question \ref{question:intro} in the affirmative for every space which is a homology manifold.

\begin{theorem}
	\label{theo:intro_homology_qi_R^n}
	Let $(\X,\sfd)$ be a proper, $\CAT(0)$ homology manifold. If $(\X,\sfd)$ is quasi-isometric to $\R^n$, then $\X$ is homeomorphic to $\R^n$.
\end{theorem}
 
We stress the fact that the statement of Theorem \ref{theo:intro_homology_qi_R^n} is definitely  not trivial, since 
the class of proper, $\CAT(0)$ homology manifolds contains many spaces with topological and metric properties extremely different to the ones of Riemannian manifolds of non-positive curvature. For instance, for every $n\ge 5$ there are $\CAT(0)$ topological $n$-manifolds that are not simply connected at infinity, hence not homeomorphic to $\R^n$, see \cite[Theorem 5b.1]{DavisJanuszkiewicz1991}. On the other hand, it has been proved in \cite[Theorem 1.1]{LytchakNaganoStadler2024} that every $\CAT(0)$ topological $n$-manifold with $n\le 4$ is homeomorphic to $\R^n$. This is no longer true for $\CAT(0)$ homology $4$-manifolds: for that, it is enough to consider the Euclidean cone over the Poincaré homology sphere. \linebreak  A more detailed description and more examples of $\CAT(0)$ homology and topological manifolds are presented in Section \ref{subsec:CAT_hom_top_mfds}.

\vspace{2mm}
A  general, positive  answer  to Question \ref{question:intro}, without any cocompacity or topological  assumption is the following statement, 
which is essentially contained in \cite{Nagano2022}, and which deal with   the case when $L$ is small enough, see again Section \ref{sec:Theorems_Nagano_and_cocompact}.

\begin{theorem}
	\label{theo:intro_qi_implies_biLip}
	Given $\varepsilon \! > \! 0$ and $n \in \N$, there exists $\delta \!=\! \delta(\varepsilon, n)\! >\! 0$ such that the following holds. 
	Let $(\X,\sfd)$ be a proper, geodesically complete, $\CAT(0)$ space.
    \begin{itemize}[leftmargin=7mm]
    \item[(i)] If $(\X,\sfd)$ is $(1+\delta, C)$-quasi-isometric to $\R^n$, then  $(\X,\sfd)$
      is $(1+\varepsilon)$-biLipschitz homeomorphic to $\R^n$.
    \item[(ii)] Moreover, if $(\X,\sfd)$ is $(1, C)$-quasi-isometric to $\R^n$, then $(\X,\sfd)$ is isometric to $\R^n$.
    \end{itemize}
\end{theorem}
 
A sharper,  general affirmative answer to Question \ref{question:intro} is provided by the following result, that highlights again the crucial role of the multiplicative constant $L$ of the quasi-isometry.

\begin{theorem}
    \label{theo:intro_2}
    If $(\X,\sfd)$ is a proper, geodesically complete, $\CAT(0)$ space  which is $(L,C)$-quasi-isometric to $\R^n$ for $L<2^{1/n}$, then $\X$ is 
    homeomorphic to $\R^n$.
\end{theorem}
We stress that Theorem \ref{theo:intro_homology_qi_R^n} is a fundamental step for Theorem \ref{theo:intro_2}, as we will explain  in \ref{sec:outline}. \\ 
Let us briefly explain how the constant $2^{1/n}$ appears here.
In Section \ref{subsec:CAT_QI_R^n}, we will prove that if  a proper, geodesically complete, $\CAT(0)$ space  $(\X,\sfd)$    is $(L,C)$-quasi-isometric to $\R^n$, then the Euclidean cone 
over the Tits boundary $\partial_T\X$ is $L$-biLipschitz homeomorphic to $\R^{n}$. This leads to the inequality
\begin{equation}
    \label{eq:biLip_Tits_sphere_volume}
    \mathcal{H}^{n-1}(\partial_T\X) \le L^{n} \cdot \mathcal{H}^{n-1}(\mathbb{S}^{n-1}).
\end{equation}
Then, the reason for the constant $2^{1/n}$ appearing in Theorem \ref{theo:intro_2}   lies in the following deep result, of which Theorem \ref{theo:intro_2} is a direct consequence.
\begin{theorem}
    \label{theo:intro_2_volume}
    Let $(\X,\sfd)$ be a proper, geodesically complete, $\CAT(0)$ space which is quasi-isometric to $\R^n$. If $\mathcal{H}^{n-1}(\partial_T\X) <2 \cdot \mathcal{H}^{n-1}(\mathbb{S}^{n-1})$ then $\X$ is 
homeomorphic to $\R^n$.
\end{theorem}

\vspace{2mm}
The last main result of our paper is to show  that Question \ref{question:intro} has a \emph{negative} answer in general, 
 and that Theorem \ref{theo:intro_2_volume} is sharp.

\begin{theorem}
\label{theo:intro_negative_answer}
      For any  $k\in \N^\ast$  there exists a proper, geodesically complete, purely $n$-dimensional, $\CAT(0)$  simplicial complex $(\X_k,\sfd_k)$ which is $(2,\frac{1}{k})$-quasi-isometric to $\R^n$, but   is not a homology manifold. Moreover, $\mathcal{H}^{n-1}(\partial_T\X_k) = 2\cdot\mathcal{H}^{n-1}(\mathbb{S}^{n-1})$ for every $k$.
\end{theorem} 

Finally, we  point out that the sequence of spaces $(\X_k,\sfd_k)$ of  Theorem \ref{theo:intro_negative_answer} converge in the pointed Gromov-Hausdorff sense to a proper, $\CAT(0)$ topological manifold. This shows the following result.

\begin{corollary}
\label{cor:intro_hom_mfd_not_open}
    The class of proper, $\CAT(0)$ homology (or topological) manifolds is not an open subset of the class of all proper, geodesically complete  $\CAT(0)$ spaces, with respect to the Gromov-Hausdorff topology.
\end{corollary}

\subsection{Outline of the proofs}\label{sec:outline}${}$\\
Our main contribution is in the proof of Theorems \ref{theo:intro_homology_qi_R^n}, \ref{theo:intro_2}, \ref{theo:intro_2_volume}, \ref{theo:intro_negative_answer} and Corollary \ref{cor:intro_hom_mfd_not_open}; here, we will focus on Theorem \ref{theo:intro_homology_qi_R^n} and \ref{theo:intro_2_volume} to explain the main ideas in the paper. 
\vspace{1mm}

The proof of Theorem \ref{theo:intro_homology_qi_R^n} is based on a combination of topological and geometric arguments. 
Because of classical topological results, Theorem \ref{theo:intro_homology_qi_R^n} is equivalent to the following statement: {\em every proper $\CAT(0)$ homology manifold $\X$ which is quasi-isometric to $\R^n$ is a topological $n$-manifold which is simply connected at infinity} (see \cite{Stallings1962} for $n\ge 5$, \cite{Freedman2014} for $n=4$, and \cite{Brin1989} for $n=3$).  
 So, we have to show that $\X$ is a topological $n$-manifold and   is simply connected at infinity.
 
A deep theorem of Lytchak-Nagano (cp. \cite[Theorem 1.1 and Proposition 2.3]{LN-finale-18}) says that a proper, geodesically complete, $\CAT(0)$ space is a topological $n$-manifold if and only if for every $x\in \X$, every small enough metric sphere $S_r(x)$ is homotopy equivalent to $\mathbb{S}^{n-1}$.

On the other hand, as $(\X,\sfd)$ is quasi-isometric to $\R^n$, we deduce that it has  cone at infinity  $C_\infty \X$ 
  bi-Lipschitz homeomorphic to $\R^n$  (see Proposition \ref{prop:QI_to_R^n_has_biLip_asymptotic_cones}). 
Moreover,  $C_\infty \X$ is  isometric to the Euclidean cone $\textup{Cone} (\partial_T \X)$ over the Tits boundary $\partial_T \X$ of $\X$ (\cite[Proposition 3.4]{Nagano2022}), 
and the latter is
homotopy equivalent to the standard sphere $\mathbb{S}^{n-1}$, see again Proposition \ref{prop:QI_to_R^n_has_biLip_asymptotic_cones}.
Then, a result of Nagano (\cite[Theorem 4.6]{Nagano2022}) implies that, for every fixed $x\in \X$ and every radius $R$ big enough, the metric sphere $S_R(x)$ is homotopy equivalent to $\mathbb{S}^{n-1}$; in particular $S_R(x)$ is a simply connected, homology $(n-1)$-sphere.

Therefore, large metric spheres around every point of $\X$ are homotopy equivalent to $\mathbb{S}^{n-1}$, while we would like to have that very small metric spheres  are homotopy equivalent to $\mathbb{S}^{n-1}$.   The next step is to relate the homology of large metric spheres to the homology of small ones, and this is where the assumption of having a homology manifold  plays a role. In order to do that, we use the natural contraction map from large balls  to small balls. Using an argument of Thurston (\cite[Corollary 2.10]{Thurston1996}) we can prove that the contraction maps have acyclic fibers. This is enough to apply Vietoris-Begle's Theorem as in Davis-Januszkiewicz \cite{DavisJanuszkiewicz1991} to conclude that small metric spheres have the same homology of large metric spheres, and so they are homology spheres. Moreover, the same construction gives that small metric spheres are simply connected. But a simply connected homology sphere is homotopy equivalent to $\mathbb{S}^{n-1}$, which concludes the proof of the fact that $\X$ is a topological manifold.

The simply connectedness at infinity of $\X$ essentially comes from the same argument, since large metric spheres are simply connected.

\vspace{1mm}
We remark that the above argument cannot be applied to  a proper, geodesically complete, $\CAT(0)$ space which is not a homology manifold. Indeed, the spaces constructed in Theorem \ref{theo:intro_negative_answer} have points $x_k \in \X_k$ and radii $0<r<R$ with the following properties:
\vspace{-3mm}

\begin{itemize}[leftmargin=7mm]
    \item one fiber of the contraction map $\phi^R_r \colon B_R(x_k) \to B_r(x_k)$ centered at $x_k$ is not acyclic, actually it is not connected;
    \item the metric sphere $S_R(x_k)$ is homotopy equivalent to $\mathbb{S}^{n-1}$, while  $S_r(x_k)$ is not;
    \item there exists no continuous map $\sigma\colon B_r(x_k) \to B_R(x_k)$ such that $\phi^R_r \circ \sigma = \text{id}$, i.e. there is no way to extend continuously geodesics.
\end{itemize}

\vspace{2mm}
On the other hand, the proof of Theorem \ref{theo:intro_2_volume} is based on a combination topological, geometric and analytical arguments. As discussed above, a proper, geodesically complete, $\CAT(0)$ space which is quasi-isometric to $\R^n$ has a cone at infinity isometric to the Euclidean cone over the Tits boundary $\partial_T\X$, and $\partial_T\X$ is a $(n-1)$-homology manifold homotopy equivalent to $\mathbb{S}^{n-1}$. Our goal is to show that $\X$ is a homology $n$-manifold. By a characterization proved in \cite{LN-finale-18}, see Proposition \ref{prop:characterization_hom_manifolds}, it is enough to show that each space of direction $\Sigma_x\X$, $x\in \X$, has the same homology of $\mathbb{S}^{n-1}$. For every $x\in \X$ we have the natural $1$-Lipschitz surjective map $\partial\log_x\colon \partial_T\X \to \Sigma_x\X$ that sends a point at infinity $[\mathfrak{r}]$ to the starting direction of the unique geodesic ray issuing at $x$ and pointing towards $[\mathfrak{r}]$. The conclusion then follows by the next result, applied to $\Y=\partial_T\X$, $\Y'=\Sigma_x\X$ and $f=\partial\log_x$.

\begin{theorem}
    \label{theo:intro_sphere_degree}
    Let $f\colon (\Y,\sfd) \to (\Y',\sfd')$ be a $1$-Lipschitz, surjective map between compact, geodesically complete, $\CAT(1)$ spaces of dimension $n$.
    Assume that $\Y$ is a homology manifold homotopy equivalent to $\mathbb{S}^n$, with $\mathcal{H}^n(\Y) < 2\cdot\mathcal{H}^{n}(\mathbb{S}^n)$: then,  $\Y'$ is also a homology $n$-manifold which is homotopy equivalent to $\mathbb{S}^n$.
\end{theorem}

The proof of Theorem \ref{theo:intro_sphere_degree} is by induction on the maximal spherical factor that $\Y'$ splits in its unique join decomposition, see \cite[Corollary 1.2]{Lytchak2005}. Again, to prove that $\Y'$ is a homology manifold it is enough to show that $\Sigma_{y'}\Y'$ has the same homology of $\mathbb{S}^n$ for every $y' \in \Y'$. \\
For this, we define a natural $1$-Lipschitz, surjective map $g_{y'}\colon\Y' \to \Sigma_{y'}\Y' \ast \mathbb{S}^0$ in order to use the induction hypothesis on the map $g_{y'}\circ f\colon \Y \to \Sigma_{y'}\Y' \ast \mathbb{S}^0$, and conclude that $\Y'$ is a homology manifold. 
The control of the homotopy type 
of $\Y'$ is more delicate.
As recalled before, it is enough to prove that $\Y'$ is a simply connected homology sphere. In order to do this,  we relate the topological degree between oriented homology manifolds to the Jacobian of maps between rectifiable metric spaces, and then use the area  formula to infer that $\Y'$ is simply connected 
and orientable. Moreover, the map $f$ satisfies
$$2\cdot \mathcal{H}^n(\mathbb{S}^{n}) >\mathcal{H}^n(\Y) \ge \vert \textup{deg}(f)\vert \cdot \mathcal{H}^n(\Y') \ge \mathcal{H}^n(\mathbb{S}^{n}),$$
(where the last inequality is a standard property of geodesically complete, $\CAT(1)$ spaces), so applying again the  area  formula yields that $f$ must have degree one. Now, the fact that $\Y'$ has the same homology of $\mathbb{S}^n$ follows by Poincaré duality. 

\subsection{Relation to the literature}${}$\\
The proof of Theorem \ref{theo:intro_homology_qi_R^n} uses crucially the deep analysis of \cite{Thurston1996}. Recently, the same ideas have been used in \cite{LytchakNaganoStadler2024} to prove that $\CAT(0)$ topological $4$-manifolds are homeomorphic to $\R^4$. We also use intensively some ideas of \cite{LN-finale-18}, especially when we need to improve infinitesimal topological properties to global ones.
\vspace{1mm}

\noindent Theorem \ref{theo:intro_2_volume}
  might be compared with \cite[Theorem 1.2]{Nagano2022} where (building on \cite{LN19})  it is proved that if  a proper, geodesically complete CAT$(0)$-space $X$ satisfies  $$\mathcal H^{n-1} (\partial_T X) < \tfrac32 \cdot  \mathcal H^{n-1} (\mathbb S^{n-1})$$
  then $X$ is homeomorphic to $\R^n$, {\em provided   that} $(\X,\sfd)$ is purely $n$-dimensional. On the other hand, in our statement, besides the better constant $2$,    pure $n$-dimensionality is part of  the conclusion of the theorem.
The proof of this seemingly innocent conclusion requires the assumption on the volume of the Tits boundary: we do not know if the same is true  
for every proper, geodesically complete CAT$(0)$ space which is quasi-isometric to $\R^n$, cp. Remark \ref{rmk:2-dim-qi_implies_pure_dimensional}.

\vspace{1mm}
\noindent   Theorem \ref{theo:intro_2_volume} also shares similarities with \cite[Theorem 1.2]{Nag18}, where it is proved that there exists $c>3/2$ such that any $\CAT(1)$ homology $n$-manifold  $(\Y,\sfd)$ satisfying $\mathcal{H}^{n}(\Y) \le c\cdot \mathcal{H}^n(\mathbb{S}^n)$ is homeomorphic to $\mathbb{S}^n$. 



{\bf Acknowledgments.}
The authors thank A. Lytchak for many interesting discussions and N. Vikman for his help on a preliminary version of Theorem \ref{theo:intro_sphere_degree} in dimension $n=2$, which is sketched in the Appendix and that may help the reader to visualize the general case.

\vspace{2mm}
\section{Preliminaries}

Throughout this paper, we will denote the standard $n$-dimensional Euclidean space   by $\R^n$,  the round $n$-sphere  by $\mathbb{S}^n$ and the $n$-dimensional Hausdorff measure by $\mathcal{H}^n$.

Let $(\X,\sfd)$ be a metric space. The open (resp. closed) ball of center $x\in \X$ and radius $r > 0$ is denoted by $B_r(x)$ (resp. $\overline{B}_r(x)$). The metric sphere around $x\in \X$ of radius $r>0$ is denoted by $S_r(x)$.  Given $\Y\subseteq \X$ and $r>0$, we denote by $\Cov(\Y,r) \in \N \cup \{\infty\}$ the minimal number of balls of radius $r$ needed to cover $\Y$. The metric space is said to be {\em metrically doubling} if there exists $C_D \ge 1$ such that $\Cov(B_{2r}(x),r)\le C_D$ for every $x\in \X$ and every $r>0$. 
Finally, we will write $(\X,\sfd) \cong  (\X',\sfd')$ for two isometric metric spaces.

A {\em geodesic segment} is an isometric embedding $\mathfrak{g} \colon [a,b] \to \X$ for some compact interval $[a,b] \subset \R$.
A geodesic {\em ray} is an isometric embedding $\mathfrak{r}\colon [0,\infty) \to \X$, and a geodesic {\em line} is an isometric embedding  $\mathfrak{g} \colon \R \to \X$. A {\em reparametrized} geodesic is the composition of a geodesic $\mathfrak{g}$ with a monotone reparametrization of $[a,b]$.
A {\em local} geodesic is a map $\mathfrak{g}\colon [a,b]\to \X$ such that for every $t\in [a,b]$ there exists $\varepsilon > 0$ such that the restriction of $\mathfrak{g}$ to $[a,b]\cap (t-\varepsilon,t+\varepsilon)$ is a geodesic segment. 

A metric space $(\X,\sfd)$ is {\em geodesically complete} if every local geodesic can be extended, i.e. if for every local geodesic $\mathfrak{g} \colon [a,b] \to \X$
there exists $\varepsilon > 0$ and a local geodesic  $\tilde{\mathfrak{g}} \colon [a-\varepsilon,b+\varepsilon] \to \X$ such that the restriction of $\tilde{\mathfrak{g}}$ to $[a,b]$ coincides with $\mathfrak{g}$. 
A metric space $(\X,\sfd)$ is {\em $R$-geodesic}, for $R\in [0,+\infty]$, if for every $x,y\in \X$ with $\sfd(x,y)<R$ there exists a geodesic segment joining the two. If $R=\infty$ we say that $(\X,\sfd)$ is geodesic.

The \emph{Euclidean cone over $\X$} is the space $\textup{Cone}(\X) = [0,+\infty) \times \X$ modulo the equivalence relation that identifies all the points $(0,x)$ for every $x\in \X$, equipped with the metric
\begin{equation}
    \label{eq:metric_cone_definition}
    \sfd^2((t,x),(t',x')) = t^2 + t'^2 - 2tt'\cos(\min\{\pi, \sfd_\X(x,x')\}).
\end{equation}
The \emph{section} of the Euclidean cone $\textup{Cone}(\X)$ is the metric space $(\X, \min\{\sfd_\X, \pi\})$, while its \emph{vertex} $o$ is the equivalence class of the points $(0,x)$, $x\in \X$.

We will assume in the following that the reader is familiar with the concepts of Gromov-Hausdorff convergence and ultralimits. Among the many references in the literature we   mainly refer to \cite[Section 3]{Cav21ter}, where precise definitions are given and a comparison between the two convergences is made. A similar discussion,  which does not take into account group actions, can be found in \cite{Jan17} (unpublished).

\vspace{2mm}
\subsection{Geodesically complete, $\CAT(\kappa)$ spaces} ${}$\\
Let $\kappa \in \R$ and let $D_\kappa := \pi/\sqrt{\kappa}$. For basics of $\CAT(\kappa)$ spaces we refer to \cite{BH09}.  Here, we will recall some properties of proper, geodesically complete, $\CAT(\kappa)$ spaces. The main reference for this part is \cite{LN19}. We recall that a $\CAT(\kappa)$ space is $D_\kappa$-geodesic.

Let $(\X,\sfd)$ be a proper, geodesically complete, $\CAT(\kappa)$ space.   For every $x\in \X$ there exists an integer $n_x \in \N$ such that the Hausdorff dimension of every small enough ball $B_r(x)$ is $n_x$. The integer $n_x$ is called the {\em dimension of $\X$ at $x$} and it is denoted by $\dim(x;\X)$. The dimension of $\X$ is correspondingly defined as
$$\dim(\X) := \sup_{x\in \X} \dim(x;\X).$$
The space $(\X,\sfd)$ is said to be {\em pure-dimensional} if $\dim(x;\X) = \dim(\X)$ for every $x\in \X$. It is said to be purely $n$-dimensional if it is pure-dimensional and $\dim(\X) = n$.

Let $x\in \X$. The {\em space of directions} at $x$ is 
$$\Sigma_x\X := \{\mathfrak{g}\colon [0,\varepsilon] \to \X\,:\, \mathfrak{g} \text{ geodesic with } \mathfrak{g}(0)=x\}/\sim,$$
where $\mathfrak{g} \sim \mathfrak{g'}$ if and only if $\angle_x(\mathfrak{g},\mathfrak{g}') = 0$; for the definition of the {\em Alexandrov angle} between two   geodesics, we refer to \cite{BH09}. The space $\Sigma_x\X$ equipped with the distance $\angle_x(\cdot,\cdot)$ is a compact, $\CAT(1)$ space of diameter $\pi$. The Euclidean cone over the space of directions, namely $T_x\X := \textup{Cone}(\Sigma_x\X)$ is a proper $\CAT(0)$ space which is called the {\em tangent space} at $x$. It is the Gromov-Hausdorff limit of every sequence of pointed metric spaces $(\X,\lambda_i\sfd,x)$ with $\lambda_i \to \infty$.
A deep result of \cite{LN19} says that the space of directions is homotopy equivalent to small enough metric spheres:

\begin{prop}\cite[Theorem 1.12]{LN19}
\label{prop:homotopy_type_small_metric_spheres}
    Let $(\X,\sfd)$ be a proper, geodesically complete, $\CAT(\kappa)$ space. For every $x\in \X$ there exists $r_x > 0$ such that $\Sigma_x\X$ is homotopy equivalent to $S_r(x)$ for every $0<r<r_x$.
\end{prop}

Throughout the paper we will be mainly interested to the case $\kappa = 0$.  However,  we will also need to consider $\CAT(1)$ spaces like the spaces of directions and the Tits boundary of a $\CAT (0)$-space. Regarding the special case $\kappa=0$, we recall that in every geodesically complete CAT$(0)$ space $\X$, for every $x \in \X$ and   $0<r\le R$, there exists a well-defined   {\em contraction map}:
\begin{equation}
    \label{eq:defin_contraction_map}
    \phi^{R,r}_x:   \overline{B}_R(x) \to \overline{B}_r(x)
\end{equation}
obtained by sending a point $y \in \overline{B}_R(x)$ to the unique point $y'$ along the geodesic $[x,y]$ satisfying  $ d(x,y') = \frac{r}{R} d(x,y)  $; we will  use the notation $\varphi^{R,r}_x$  for the restriction of the above map to the sphere $S_R(x)$.
By the geodesic completeness  of $\X$  and the  $\CAT(0)$ assumption, the map $\phi^{R,r}_x$ is surjective and  $\frac{r}{R}$-Lipschitz,  see  \cite[Proposition 5.1]{LN19} or \cite[Lemma 2.1]{CavS20}.  This implies that the {\em logarithmic map} 
$$\log_x : \X \to \Sigma_x \X$$ 
sending any $y \in \X$ to the  initial direction of the geodesic $[x,y]$, is $1$-Lipschitz and surjective, cp. \cite[Lemma 2.3]{CavS20}.

\vspace{2mm}
\subsection{Tits boundary and asymptotic cone}
\label{sec:titsboundary}${}$\\
Let $(\X,\sfd)$ be a proper, geodesically complete, $\CAT(0)$ space. Two geodesic rays $\mathfrak{r}, \mathfrak{r}' $
are said to be  equivalent if $\sup_{t\ge 0} \sfd(\mathfrak{r}(t), \mathfrak{r}'(t)) < \infty$. This is an equivalence relation on the set of geodesic rays. The resulting space is called the {\em ideal boundary} of $\X$ and it is denoted by $\partial \X$. For every $x\in \X$ and every geodesic ray $\mathfrak{r}$, there exists a unique geodesic ray $\mathfrak{r}_x$ which is asymptotic to $\mathfrak{r}$ and such that $\mathfrak{r}_x(0) = x$   (see \cite[Proposition II.8.2]{BH09}). Given two elements $[\mathfrak{r}],[\mathfrak{r}'] \in \partial \X$ and $x\in \X$ we define
$$\angle([\mathfrak{r}],\mathfrak{[r']}) := \sup_{x\in \X} \angle_x(\mathfrak{r}_x, \mathfrak{r}_x').$$
This defines a distance on $\partial\X$, called {\em angular metric}.  The Tits boundary of $\X$ is the set $\partial \X$ equipped with the length metric induced by the angular metric $\angle(\cdot,\cdot)$ and it is denoted by $\partial_T\X$. It is a $\CAT(1)$ space, see \cite[Theorem II.9.20]{BH09}). 

For every $x\in \X$, the map
\begin{equation}
    \label{eq:defin_log_from_Tits_boundary}
    \partial\log_x\colon \partial_T\X \to \Sigma_x\X
\end{equation}
assigning to each $[\mathfrak{r}] \in \partial_T\X$ the direction of the geodesic ray $\mathfrak{r}_x$ at $x$ is $1$-Lipschitz, since
$$\sfd_T([\mathfrak{r}],[\mathfrak{r}']) \ge \angle([\mathfrak{r}],[\mathfrak{r}']) \ge \angle_x(\mathfrak{r}_x, \mathfrak{r}_x') = \angle_x(\partial\log_x([\mathfrak{r}]),\partial\log_x([\mathfrak{r}'])).$$
and surjective because of geodesic completeness.

\vspace{2mm}

Following \cite{Nagano2022}, we will say that $\X$ {\em admits the cone at infinity} if there exists a pointed metric space $(C_\infty \X, \sfd_\infty, x_\infty)$ such that for every sequence $\lambda_i \to 0$ and every  $x\in \X$, the sequence of pointed metric spaces $(\X, \lambda_i\sfd, x)$ converges in the pointed Gromov-Hausdorff sense to $(C_\infty \X,\sfd_\infty,x_\infty)$; then,  $C_\infty \X$ is called the {\em cone at infinity} of $\X$. 
Notice that if the convergence is true for every sequence $\lambda_i$ going to $0$ and some fixed point $x\in \X$, then it is also true for every sequence and every other point of $\X$. Moreover, if $C_\infty \X$ exists, then it is proper, by definition of pointed Gromov-Hausdorff convergence; so, if it exists, it is a proper, geodesically complete, $\CAT(0)$ space. 

The existence of the cone at infinity is equivalent to several natural asymptotic conditions and bounds on the growth of $\X$, as we now recall. 
\begin{prop}
	\label{prop:properties_asymptotic_cone}
	Let $(\X,\sfd)$ be a proper, geodesically complete, $\CAT(0)$ space. The following properties are equivalent:
	\begin{itemize}
		\item[(i)]  $(\X,\sfd)$ admits the cone at infinity;
		\item[(ii)]  $\omega$\textup{-}$\lim(\X,\lambda_i\sfd,x)$ is proper for some $x\in \X$, some sequence $\lambda_i \to 0$ and some non-principal ultrafilter $\omega$;
		\item[(iii)] the Tits boundary $\partial_T\X$ is compact;
		\item[(iv)] $(\X,\sfd)$ is metrically doubling.
	\end{itemize}
	If any of these conditions is satisfied, then $\omega$\textup{-}$\lim(\X,\lambda_i\sfd,x)$ is isometric to $(C_\infty \X, \sfd_\infty, x_\infty)$ for every choice of $x\in \X$, $\lambda_i \to 0$ and non-principal ultrafilter $\omega$; moreover, $C_\infty \X$ is isometric to the Euclidean cone $\textup{Cone}(\partial_T\X)$ over $\partial_T\X$, and $\dim(C_\infty\X) = \dim(\X)$. 

    \vspace{1mm}
    \noindent Furthermore, if $\X$ is purely $n$-dimensional, conditions {\em (i)-(iv)} are also equivalent to:
	\begin{itemize}
		\item[(v)] the $n$-dimensional Hausdorff measure $\mathcal{H}^n$ is doubling on $(\X,\sfd)$. 
	\end{itemize}
\end{prop}

The fact that $C_\infty \X$ is isometric to   $\textup{Cone}(\partial_T\X)$ justifies the name {\em cone at infinity}.

\begin{proof}
	Recalling the relation  between ultralimits and Gromov-Hausdorff convergence proved in \cite[Section 3]{Cav21ter}, the equivalences (i)-(ii)-(iii) simply follow from \cite[Proposition 3.4]{Nagano2022}. The fourth equivalence is also proved in \cite[Proposition 3.7]{Nagano2022}, under the assumption that $(\X,\sfd)$ has pure dimension, which is unnecessary. Actually, if $(\X,\sfd)$ is metrically doubling then  $\omega$-$\lim (\X,\lambda_i\sfd,x)$ is proper for every $x\in \X$, every $\lambda_i \to 0$ and every non-principal ultrafilter $\omega$,  by \cite[Proposition 3.13]{Cav21ter}, because it satisfies the assumptions of Gromov's precompactness Theorem. Now, suppose that $(\X,\sfd)$ is not metrically doubling. Therefore, for every $n\in \N$ there exists a point $x_n\in \X$ and a radius $r_n > 0$ such that $\Cov(B_{2r_n}(x_n),r_n) > n$. 
    We fix $x\in \X$. By \cite[Proof of Lemma 4.5]{CavS20} we can suppose without loss of generality that $r_n\ge 1$, so that $\sfd(x_n,x) =: d_n \to \infty$. We set $\lambda_n := d_n^{-1}$ and we observe that $\Cov(B_{2d_n}(x_n),d_n) > n$ again by \cite[Proof of Lemma 4.5]{CavS20}. It is now easy to see that the space $\omega$-$\lim (\X,\lambda_n\sfd,x)$ is not proper.
    
    If the conditions (i)-(iv) hold then $C_\infty\X = \textup{Cone}(\partial_T\X)$ by \cite[Proposition 3.4]{Nagano2022}.
    The equality $\textup{dim}(C_\infty\X) = \textup{dim}(\X)$ comes from \cite[Proposition 6.5]{CavS20} noticing that $\textup{dim}(\X)$ is finite since $\X$ is metrically doubling, by \cite[Theorem 4.9]{CavS20}.
    
    Finally, suppose that $(\X,\sfd)$ is purely $n$-dimensional. Then $(\X,\sfd)$ is metrically doubling if and only if $\mathcal{H}^n$ is doubling by \cite[Corollary 5.5]{CavS20}.
\end{proof}

\vspace{2mm}
\subsection{$\CAT(1)$ spherical joins} \label{join}
${}$\\
    Let $(\Y,\sfd)$, $(\Y',\sfd')$ be two metric spaces. The \emph{spherical join} $\Y \ast \Y'$ is the set $[0,\pi/2] \times \Y \times \Y'$ modulo the equivalence relation which identifies $(\vartheta, y,y')$ and $(\varphi,z,z')$ if ($\vartheta=\varphi = 0$ and $y=z$) or ($\vartheta = \varphi = \pi/2$ and $y'=z'$). The equivalence class of $(\vartheta,y,y')$ will be denoted by $y\cos \vartheta + y'\sin \vartheta$ and we identify $Y, Y'$ to the subsets of $\Y \ast \Y'$ corresponding respectively to the classes of $[0,y, \cdot]$ and $[\frac{\pi}{2}, \cdot, y']$.
We equip $\Y \ast \Y'$ with the metric $\sfd_\ast$ which is at most $\pi$ and that satisfies, for given $x_1=y\cos \vartheta + y'\sin \vartheta$, $x_2= z\cos \varphi + z'\sin \varphi$,
$$\cos(\sfd_{\ast}(x_1,x_2)) = \cos\vartheta \cos \varphi \cos(\min\{\pi, \sfd_\Y(y,z)\}) + \sin\vartheta \sin \varphi \cos(\min\{\pi, \sfd_{\Y'}(y',z')\}).$$
Spherical joins are related to Euclidean cones in the following way.
\begin{lemma}[{\cite[Proposition I.5.15]{BH09}}]
\label{lemma:cones_and_joins}
    Let $(\Y,\sfd_\Y), (\Y',\sfd_{\Y'})$ be two metric spaces. \\ Then $\textup{Cone}(\Y \ast \Y')\cong \textup{Cone}(\Y) \times \textup{Cone}(\Y')$.
\end{lemma}
In particular, we have $\mathbb{S}^n \ast \mathbb{S}^m \cong \mathbb{S}^{n+m+1}$ for every $n,m\in \N$.

In the sequel we will need some basic facts about the geometry of spherical joins and Euclidean cones.
Recall that a Euclidean cone is $\CAT(0)$ if and only if its section is $\CAT(1)$, see \cite[Proposition II.3.14]{BH09}, and it is geodesically complete if and only if its section is geodesically complete by \cite[Proposition I.5.10]{BH09}.



%

\begin{lemma}
\label{lemma:tangent_to_section_of_cone}
    Let $\X = \textup{Cone}(\Y)$ be the Euclidean cone over a  $\CAT(1)$  space $(\Y,\sfd_\Y)$. For every $y\in \Y$ we have $T_{(1,y)}\X \cong \R\times T_y\Y$ and   $\Sigma_{(1,y)}\X \cong \mathbb{S}^0 \ast \Sigma_y\Y$.
\end{lemma}

\begin{proof}
First, $T_{(1,y)}  \textup{Cone}(\Y)  =  T_{(1,y)} \left( \R   \times \! \Y \right)$, for $\R \times \! \Y $ endowed with the product metric $\sfd_\times$. \\
In fact,   any geodesic   from     $x=(1,y)$ in  $\textup{Cone}(\Y)\!\setminus\! \{o\}$ or in $ \R \times \! \Y  $   can be reparameterized as 
 $(t(s),y(s))$  where  $y(s)$ is a geodesic on $\Y$ (cp. \cite[Prop. I.5.3 and I.5.10]{BH09}) and $t(0)= 1$.
   Now, for any germ of curve $\gamma (s) = (t(s),y(s))$  we have, with respect to the cone metric, 
$$\sfd (\gamma (s), x) = 1+t(s)^2 -2t(s) \cos \sfd_\Y (y(s),y)
= 1+t(s)^2 -2t(s) +t(s)  \sfd_\Y (y(s),y)^2 + t(s) O(s^4)$$
which,  for $s \to 0$,  behaves like  
$\sfd_\times (\gamma (s), x) = (t(s) -1)^2 + \sfd_\Y (y(s),y)^2$,
by definition of product metric $\sfd_\times $.  
Then,  for any two such curves $\gamma $ and $ \gamma' $ from $x$, their (upper) Alexandrov angle 
$\angle_x(\gamma,\gamma') $   is the same when measured in $(\X,\sfd) $ or in $(\R \times  \Y , \sfd_\times )$.
This implies that  $\Sigma_{(1,y)} \textup{Cone}(\Y) =  \Sigma_{(1,y)} \left( \R   \times  \Y \right)$, hence that  
 $T_{(1,y)}  \textup{Cone}(\Y)  =  T_{(1,y)} \left( \R   \times  \Y \right)$.\\
Moreover, it is well-known that,  for any metric  product  $\Y' \times \! \Y$, we have
$$\Sigma_{(y',y)} (\Y' \times  \Y) 
\cong \Sigma_{y'} \Y'  \ast   \Sigma_y \Y 
\hspace{2mm} \mbox{ and }   \hspace{2mm} 
T_{(y',y)} (\Y' \times  \Y) \cong T_{y'} \Y'  \times   T_y \Y $$  
therefore $\Sigma_{(1,y)}   \textup{Cone}(\Y)  \cong \Sigma_{1} \R  \ast   \Sigma_y \Y \cong \mathbb{S}^0\ast \Sigma_y\Y $, and in turns $T_{(1,y)}\X \cong \R\times T_y\Y$. 
\end{proof}

A consequence of the above lemmas is the computation of some spaces of directions of a spherical join.

\begin{lemma}
\label{lemma:space_of_directions_to_spherical_join}
    Let $(\Y,\sfd)$, $(\Y',\sfd')$ be two compact, geodesically complete, $\CAT(1)$ spaces and consider the spherical join $\Y\ast \Y'$. For every $y\in\Y$ it holds $\Sigma_{y}(\Y\ast\Y')\cong \Sigma_y\Y \ast \Y'$.
\end{lemma}
\begin{proof}
    The point $y\in \Y$ can be seen as the point $(0,y,\cdot) \in \Y\ast\Y'$. Applying several times Lemma \ref{lemma:cones_and_joins} and Lemma \ref{lemma:tangent_to_section_of_cone} we obtain that
    \begin{equation*}
    \begin{aligned}
        \R \times \textup{Cone}(\Sigma_y(\Y\ast \Y')) \cong \R \times T_y(\Y\ast \Y') &\cong T_{(1,y)}\textup{Cone}(\Y\ast \Y') \\
        &\cong T_{(1,y)}\textup{Cone}(\Y) \times T_{o'}\textup{Cone}(\Y') \\
        &\cong \R \times \textup{Cone}(\Sigma_y\Y) \times \textup{Cone}(\Y'),
    \end{aligned}
    \end{equation*}
    where $o'\in \textup{Cone}(\Y')$ denotes the vertex. Hence,
    $$\textup{Cone}(\Sigma_y(\Y\ast \Y')) \cong \textup{Cone}(\Sigma_y\Y) \times \textup{Cone}(\Y')\cong\textup{Cone}(\Sigma_y\Y\ast \Y')$$
since the $\R$-factor clearly is sent into the $\R$-factor by the above isometries (see also \cite[Theorem 1.1]{FoertschLytchak2008}) and by Lemma \ref{lemma:cones_and_joins}. We conclude that $\Sigma_y(\Y\ast\Y') \cong \Sigma_y\Y \ast \Y'$.  
\end{proof}

Another application is the following useful fact, which will be crucial to perform induction in the proof of Theorem  \ref{theo:intro_sphere_degree}.

\begin{cor}
\label{cor:map_Y_to_S^0_join_Sigma_y}
    Let $(\Y,\sfd)$ be a geodesically complete  $\CAT(1)$ space of diameter $\pi$ and   $y\in \Y$. Then, there exists a $1$-Lipschitz surjective map $g\colon \Y \to \mathbb{S}^0\ast\Sigma_y\Y$.
\end{cor}
\begin{proof}
    We consider the space $\X := \textup{Cone}(\Y)$ and the point $ x=(1,y)\in \textup{Cone}(\Y)$. 
    Lemma \ref{lemma:tangent_to_section_of_cone} implies that $\Sigma_{x}\textup{Cone}(\Y) \cong \mathbb{S}^0 \ast \Sigma_y\Y$. Since $\X$ is a $\CAT (0)$ Euclidean cone (as $\Y$ is $\CAT(1)$), it admits the cone at infinity and its Tits boundary $\partial_T\X$ is isometric to $(\Y,\sfd)$. Then, the $1$-Lipschitz, surjective map $\partial \log_{x} \colon \partial_T\X \cong \Y \to \Sigma_{x}\textup{Cone}(\Y) \cong \mathbb{S}^0 \ast \Sigma_y\Y$ defined in \eqref{eq:defin_log_from_Tits_boundary} gives the thesis.
\end{proof}

\vspace{2mm}
\subsection{Strainer maps and local contractibility of spheres}  ${}$\\
Strainer maps play a crucial role in the description of the geometry of geodesically complete spaces with upper curvature bounds, see \cite{LN19}. In the sequel we will need only the case of $(1,\delta)$-strainer maps in $\CAT(0)$ spaces, so we will focus on this case.

\begin{defin}
    Let $(\Sigma,\sfd)$ be $\CAT(1)$ space, and let $\delta > 0$. A point $v\in \Sigma$ is called {\em $\delta$-spherical} if there exists $\bar{v}\in \Sigma$ such that
   \begin{equation}
       \label{eq:defin_delta_spherical}
       \sfd(v,w)+\sfd(w,\bar{v}) < \pi + \delta
   \end{equation}
   for every $w\in \Sigma$. The points $v$ and $\bar{v}$ are said to be $\delta$-spherical opposite.
\end{defin}

The condition above implies that every two {\em antipodes} of $v$, i.e. points with distance $\pi$ from $v$, are at mutual distance at most $2\delta$.

\begin{defin}
    Let $(\X,\sfd)$ be a proper, geodesically complete  $\CAT(0)$ space, let $p\in \X$, $\delta > 0$.   We say that the distance map $\sfd_p\colon \X \to \R$, $x\mapsto \sfd(p,x)$ is {\em $(1,\delta)$-strainer} on a subset $U\subseteq \X$ if for every $x\in U$ the direction $\log_x(p)$ is $\delta$-spherical in  $\Sigma_x\X$.
\end{defin}

The only example of $(1,\delta)$-strainer map we will use is the distance from the vertex of an Euclidean cone:
\begin{ex}
\label{ex:strainer_map_cones}
    Let $(\Y,\sfd_\Y)$ be a geodesically complete, $\CAT(1)$ metric space. Consider  $(\X,\sfd):=\textup{Cone}(\Y)$ and let $o\in \X$ be the vertex. Then $\sfd_o$ is a $(1,\delta)$-strainer map on $U=\X\setminus\{o\}$, for every $\delta > 0$. Indeed, for every $x\in \X\setminus\{o\}$, the direction $\log_x(o) \in \Sigma_x\X$ has a unique antipode.
\end{ex}

Fibers of strainer maps are locally contractible and locally path connected.

\begin{prop}
\label{prop:fibers_strainer_maps_loc_contr_loc_path_conn}
    Let $(\X,\sfd)$ be a proper, geodesically complete, $\CAT(0)$ space, let $p\in \X$ and $0 < \delta \le 1/20$. Suppose $\sfd_p\colon \X \to \R$ is $(1,\delta)$-strainer on  $U\subseteq \X$. Then, for  every $t\in [0,+\infty)$ the set $\sfd_p^{-1}(t)\cap U = S_t(p)\cap U$ is locally contractible and locally path connected.
\end{prop}
\begin{proof}
    By \cite[Theorem 9.1]{LN19}, for every $x\in S_t(p)\cap U$ we can find $\varepsilon_x>0$ and a continuous map $G\colon B_{\varepsilon_x}(x) \to B_{\varepsilon_x}(x) \cap S_t(p)$ with $G(x)=x$ and $\sfd(G(y),x)\le \sfd(y,x)$ for every $y\in B_{\varepsilon_x}(x)$. We now define the map 
    $$\Psi\colon (B_{\varepsilon_x}(x) \cap S_t(p))\times [0,\varepsilon_x] \to B_{\varepsilon_x}(x) \cap S_t(p)$$
    sending $(y,t)$ to $G(\phi_x^{\varepsilon_x,(1-t)\varepsilon_x}(y))$, where the maps $\phi$ are the contraction maps defined in \eqref{eq:defin_contraction_map}. By construction, $\Psi$ is continuous, $\Psi(\cdot,0) = \textup{id}$ and $\Psi(\cdot,1) = x$. This shows that $B_{\varepsilon_x}(x) \cap S_t(p)$ is contractible. Moreover, the path $t\mapsto \Psi(y,t)$ is contained in $B_{\varepsilon_x}(x) \cap S_t(p)$, proving that every point $y \in B_{\varepsilon_x}(x) \cap S_t(p)$ can be connected via a path in $B_{\varepsilon_x}(x) \cap S_t(p)$ to $x$. Therefore, $B_{\varepsilon_x}(x) \cap S_t(p)$ is path connected.
\end{proof}

\begin{obs}{\em
    A more refined version of Proposition \ref{prop:fibers_strainer_maps_loc_contr_loc_path_conn} has been proved in \cite[Proposition 8.3]{Nagano2026}, but we do not need it here. 
    }
\end{obs}

As an application of the stability of strainer maps under Gromov-Hausdorff convergence we obtain the following result.

\begin{prop}
\label{prop:spheres_loc_contr_and_loc_path_connected}
    Let $(\X,\sfd)$ be a proper, geodesically complete, $\CAT(0)$ space and let $x\in \X$.
    \begin{itemize}
        \item[(i)] There exists $r_x>0$ such that $S_r(x)$ is locally contractible and locally path connected for every $0<r<r_x$.
        \item[(ii)] If $(\X,\sfd)$ admits the cone at infinity then there exists $R_x>0$ such that $S_{R}(x)$ is locally contractible and locally path connected for every $R>R_x$.
    \end{itemize}
\end{prop}
\begin{proof}
    The tangent space $(T_x\X, o_x)$  is  a metric cone with base point given by the vertex $o_x$,  and it is the pointed, Gromov-Hausdorff limit of the sequence $(\X,\lambda_i\sfd,x)$ for every $\lambda_i \to \infty$. The distance map from $o_x$ is $(1,\delta)$-strainer for every $\delta > 0$, as explained in the Example \ref{ex:strainer_map_cones}. Now, \cite[Theorem 13.1]{LN19} implies that $\sfd_x \colon \X \to \R$ is $(1,\delta)$-strainer on a neighbourhood of $S_r(x)$, for $r$ sufficiently small. So, (i) follows from Proposition \ref{prop:fibers_strainer_maps_loc_contr_loc_path_conn}. 
    
    If $(\X,\sfd)$ admits the cone at infinity, then the sequence of spaces $(\X,\lambda_i\sfd,x)$ converges in the Gromov-Hausdorff sense to $C_\infty\X$ for every $\lambda_i \to 0$. Proposition \ref{prop:properties_asymptotic_cone} says that $C_\infty\X$ is a metric cone. Arguing exactly as before we conclude the proof of   (ii).
\end{proof}

    The stability of strainer maps has been also used in \cite[Theorem 1.12]{LN19} to deduce homotopy properties of small spheres. In the same way, homotopy properties of large spheres in case of the existence of the cone at infinity have been deduced in \cite{Nagano2022}:

\begin{prop}[{\cite[Theorem 4.6]{Nagano2022}}]
    \label{prop:homotopy_stability_large_spheres}
    Let $(\X,\sfd)$ be a proper, geodesically complete, $\CAT(0)$ space admitting the cone at infinity. For every $x\in \X$ there exists $R_x > 0$ such that $S_R(x)$ is homotopy equivalent to $\partial_T\X$ for every $R\ge R_x$. 
\end{prop}

    \vspace{2mm}
\subsection{$\CAT(\kappa)$ homology and topological manifolds}
\label{subsec:CAT_hom_top_mfds}
${}$

A homology $n$-manifold (over $\Z$) is a Hausdorff, locally compact topological space $\X$ with finite $\Z$-cohomological dimension (see Definitions 16.3-16.7 in \cite{bredon-book})
such that
$H_\ast(\X,\X  \setminus  \{x\}) \simeq H_\ast(\R^n, \R^n  \setminus \{0\})$ for every $x\in \X$, where $\simeq$ denotes an isomorphism of   algebraic structures.
\vspace{1mm}

Here, and everywhere in the paper, we consider Borel-Moore homology with integer coefficients and closed supports. In the definition of homology manifold, Borel-Moore homology is preferable, 
since singular homology has a more pathological behavior in general (see \cite{BarrattMilnor1962}).
Borel-Moore homology is easier to understand for Hausdorff,  locally compact and paracompact pairs $(\X, A)$ with finite $\Z$-cohomological dimension, which  are moreover locally contractible: in fact,  under these assumptions, it coincides with the (reduced) homology groups of {\em locally
finite} singular chains, see 
\cite[Corollary V.13.6 and comments in \S V.1.19]{bredon-book}. When moreover $\X$ is compact, then it coincides with the usual (reduced) singular homology (see again \cite[Ibid.]{bredon-book}).

On the other hand, $H_c^\ast(\X)$ will always denote, in the following,  Alexander-Spanier  or \v{C}ech cohomology groups with $\Z$ coefficients and compact supports \footnote{They are the same, cp. \cite[Chpt.6, Ex.D.3]{spanier}, and  coincide  with singular cohomology   when, moreover, $\X$ is   locally contractible (best, when $\X$ is homology locally connected, with respect to singular homology, see \cite[Theorem III. 2.1]{bredon-book}).}.

\begin{obs}
\label{rmk:BorelMoore_vs_Singular}
{\em
Some comments are in order.
\begin{itemize}[leftmargin=3mm,itemindent=0mm,labelwidth=\itemindent,labelsep=2mm,align=left]
    \item[(i)]  A good synthesis of some properties of homology manifolds (with or without boundary) can be found in \cite{AncelGuilbaut1999}.
In particular, according to their remark 3, p.1270,   and using \cite{Wilder1965}, every homology manifold is locally path connected. 
\vspace{1mm}

\item[(ii)] Homology manifolds are not necessarily locally contractible. Actually, a finite dimensional homology manifold is locally contractible if and only if it is an Absolute Neighbourhood Retract (ANR), which is equivalent to being {\em homotopically locally $1$-connected} (cp. \cite[remark 7, p.1270]{AncelGuilbaut1999}). If a homology manifold with boundary is an ANR, the boundary is a homology manifold (by \cite{Mitchell1990}), but it might  not be an ANR.
\vspace{1mm}

 \item[(iii)] We will apply the theory of homology manifolds to study arbitrary metric spheres of $\CAT(0)$ homology manifolds: therefore,  in this generality, we cannot assume that they are locally contractible.
\vspace{1mm}

\item[(iv)] Homology manifolds may have infinite topological dimension, but the topological dimension is always finite when they carry a  $\CAT(\kappa)$ metric, see Lemma \ref{lemma:CAT_homology_are_purely_dimensional} below.\\
As a consequence, all the homology manifolds we will consider in this paper will be finite dimensional, being subsets of  $\CAT(\kappa)$ homology manifolds.
 \vspace{1mm}   

\item[(v)] First countable homology manifolds are cohomologically locally connected over $\Z$ (clc$_\Z$ in the terminology of \cite[Def. II.17.1]{bredon-book}), see \cite{harlap} and \cite{mitchell}.
 \vspace{1mm}   

\item[(vi)] 
A Hausdorff, locally compact topological space $\X$ with finite $\Z$-cohomological dimension is called a {\em cohomology $n$-manifold} if all the local cohomology groups $H_c^{k}(\X, \X \setminus \{x\})$ are equivalent to $\Z$ in degree $n$, and vanish in degree different from $n$.\\
As a consequence of (v) and of \cite[Theorem V.16.8]{bredon-book},   for first countable spaces the definitions of homology $n$-manifold  and   cohomology $n$-manifold  (over $\Z$) coincide (see also  \cite[comment after Theorem V.16.15]{bredon-book}).
\end{itemize}
}
\end{obs}

 In \cite{LN-finale-18}, to which we frequently refer, the authors do not specify the homology theory they use. However, they use  it either for spaces of directions (which are $\CAT(1)$ hence locally contractible), or for small enough metric spheres (that are locally contractible by Proposition \ref{prop:spheres_loc_contr_and_loc_path_connected}). 
    Since moreover these spaces are compact,  classical singular homology coincides with Borel-Moore's. Therefore, Lytchak and Nagano's results recalled in this section are valid for both homology theories.   
\vspace{2mm}

 The following basic result follows from   \cite{LN19} (see also \cite[Proposition II.5.12]{BH09}).
\begin{lemma}
\label{lemma:CAT_homology_are_purely_dimensional}
    Let $(\X,\sfd)$ be a proper, $\CAT(\kappa)$ homology $n$-manifold. Then $(\X,\sfd)$ is geodesically complete and it is purely $n$-dimensional.
\end{lemma}

\begin{proof}
    The fact that $(\X,\sfd)$ is geodesically complete follows by \cite[Proposition II.5.12]{BH09}. For every $k\in \N$ we denote by $\X_k$ the set of points of $\X$ with dimension $k$. Every $\X_k$ contains a dense subset $M_k$, open in $\X$, which is a topological $k$-manifold, by \cite[Theorem 1.2]{LN19}. So, if there exists a point $x\in \X$ of dimension $k \neq n$, then $M_k \neq \emptyset$. For every $y\in M_k$ we have that $H_\ast(\X,\X\setminus\{y\})\simeq H_\ast(\R^k,\R^k\setminus \{0\})$. On the other hand, $H_\ast(\X,\X\setminus\{y\})\simeq H_\ast(\R^n,\R^n\setminus \{0\})$ because $\X$ is a homology $n$-manifold. This gives the contradiction.
\end{proof}


 Homology manifolds can be characterized among proper, geodesically complete, $\CAT(\kappa)$ spaces by looking at the space of directions.

\begin{prop}[Lemma 3.1 and Corollary 3.4, \cite{LN-finale-18}] 
\label{prop:characterization_hom_manifolds}
    Let $(\X,\sfd)$ be a proper, geodesically complete, $\CAT(\kappa)$ space. 
    Then $\X$ is a homology $n$-manifold if and only if   $H_\ast(\Sigma_x\X) \simeq H_\ast(\mathbb{S}^{n-1})$ for every $x\in \X$.  
    In this case, each $\Sigma_x\X$ is a homology $(n-1)$-manifold.
\end{prop}

Every topological $n$-manifold is a homology $n$-manifold. The vice versa is false. A  deep result of \cite{LN-finale-18} characterizes proper, geodesically complete, $\CAT(\kappa)$ spaces that are topological manifolds.

\begin{prop}
\label{prop:characterization_top_manifolds}
    Let $(\X,\sfd)$ be a proper, geodesically complete, $\CAT(\kappa)$ space. Then $\X$ is a topological $n$-manifold if and only if $\Sigma_x\X$ is homotopy equivalent to $\mathbb{S}^{n-1}$ for every $x\in \X$.
\end{prop}

The difference between Propositions \ref{prop:characterization_hom_manifolds} and \ref{prop:characterization_top_manifolds} is subtle. A topological space
 $\Y$ such that $H_\ast(\Y) \simeq H_\ast(\mathbb{S}^{n-1})$ is called a homology $(n-1)$-sphere\footnote{\label{footnote:homology_sphere}In literature, it is often required that a homology sphere is a topological manifold; we do not adhere to this convention. In particular, for us, a homology sphere  is not even a homology manifold,  a priori.}; notice that a homology $(n-1)$-sphere {\em is not} homotopy equivalent to $\mathbb{S}^{n-1}$ in general. Namely, a homology $(n-1)$-sphere $\Y$ with the   homotopy type of a CW-complex  is homotopy equivalent to $\mathbb{S}^{n-1}$ if and only if it is simply connected. Indeed, any generator   $[c]\in H_{n-1}(\Y) \simeq \Z$ gives a map $c\colon\mathbb{S}^{n-1} \to \Y$ inducing an isomorphism  $c_\ast\colon H_\ast(\mathbb{S}^{n-1}) \to H_\ast(\Y)$ in every degree, by construction; the conclusion then follows by the homology Whitehead's Theorem, see \cite[Corollary 4.3]{Hatcher2002}. 
The spaces of directions $\Sigma_x\X$ of a proper, geodesically complete $\CAT(\kappa)$ space are compact and $\CAT(1)$, so they have the homotopy type of a CW complex by \cite[Proposition II.5.13]{BH09}. 
Hence, the difference between Proposition \ref{prop:characterization_hom_manifolds} and Proposition \ref{prop:characterization_top_manifolds} relies on the fundamental group of the space of directions. 
We summarize this discussion in the next result.
\begin{prop}
\label{prop:CAT(0)_hom_mfds_to_top_mfds_simply_connected}
    A proper $\CAT(\kappa)$ homology manifold is a topological manifold if and only if every space of directions is simply connected.
\end{prop}
 
\begin{obs}
\label{obs:spaceofdirectionnothomeotosphere}
{\em
There exist $\CAT(1)$ homology manifolds $\Y$ (and even topological manifolds!) 
which have points $x$ such that $\Sigma_x\Y$ is 
homotopy equivalent,  but not homeomorphic, to $\mathbb{S}^{n-1}$, see \cite[Theorem 4]{Berestovskii1999}.   To construct a proper, geodesically complete, $\CAT(0)$ example with the same property  it is enough to consider the tangent cone $\X=T_x\Y \cong \textup{Cone}(\Sigma_x\Y)$ at $x$, with $\Y$ as above. The space $\X$ is even homeomorphic to $\mathbb{R}^n$ by \cite[Theorem 1.1]{LN-finale-18}, and  at the vertex $x$ of the cone we have   $\Sigma_x\X\cong\Sigma_x\Y$.
According to Proposition \ref{prop:characterization_hom_manifolds}, in such examples the space of directions $\Sigma_x\X$ is a $(n-1)$-homology manifold, which is homotopy equivalent to $\mathbb{S}^{n-1}$ by Proposition \ref{prop:characterization_top_manifolds}. This then implies that   $\Sigma_x\X$ is not a topological $(n-1)$-manifold, by the  solution of the (generalized) Poincaré conjecture. 
}
\end{obs}


    \begin{obs}{\em 
        There exists also   proper, geodesically complete, $\CAT(0)$ spaces $(\X,\sfd)$ (not homology manifolds) possessing a   point $x\in \X$ such that $\Sigma_x\X$ is
        homotopy equivalent to $\mathbb{S}^2$  
but  not   a homology $2$-manifold.
        This stresses the importance of the assumption that  the whole $\X$ is a homology manifold to   have the last statement in Proposition \ref{prop:characterization_hom_manifolds}.
        
        \noindent An example is produced as follows: let $\Y$ be a contractible, compact, simplicial $2$-complex with no free faces (e.g. the dunce hat, or the Bing's house) and equip it with a $\CAT(1)$ metric $\sfd_\Y$, which always exists by \cite[Theorem 1]{Berestovskii1983}. Since the local homology in degree $2$ is not trivial at every point, the space $(\Y,\sfd_\Y)$ is locally geodesically complete, see \cite[Proposition II.5.12]{BH09}.        
        Now, consider the wedge ${\rm Z} := \Y \vee \mathbb{S}^2$, which is a compact, locally geodesically complete, $\CAT(1)$ metric space. Finally, take $\X := \textup{Cone}({\rm Z})$, and let  $x\in \X$ be the vertex of the cone.  Then, $\X$ is a proper, geodesically complete, purely $3$-dimensional $\CAT(0)$ space with $\Sigma_x\X \cong {\rm Z}$  and, since ${\rm Z}$ retracts on $\mathbb{S}^2$, $\Sigma_x\X$ is homotopy equivalent to $\mathbb{S}^2$. However, $\Sigma_x\X \cong {\rm Z}$ is not a homology $2$-manifold, due to the wedge point.

        This example also shows an interesting fact:  in a proper, geodesically complete $\CAT (0)$ space $X$, the subset   $\textup{Hom}_n(\X)$ of  points $x$  such that $H_\ast(\X,\X \setminus \{x\}) \simeq H_\ast(\mathbb{S}^{n-1})$ is not open, in general.
        Indeed, the vertex $x$ belongs to $\textup{Hom}_2(\X)$ in our example, but the points of the form $(t,z) \in \X=\textup{Cone}({\rm Z})$ with $t>0$ and $z \in {\rm Z}$ the wedge point, are not in $\textup{Hom}_2(\X)$. This phenomenon is related to the statement of Corollary \ref{cor:intro_hom_mfd_not_open}.
        }
    \end{obs}

\subsection{Asymptotic  properties of CAT$(0)$ spaces quasi-isometric to $\R^n$ }
\label{subsec:CAT_QI_R^n}
${}$\\
In this section we state some properties of proper, geodesically complete, $\CAT(0)$ spaces that are quasi-isometric to $\R^n$.

\begin{prop}
	\label{prop:QI_to_R^n_has_biLip_asymptotic_cones}
	Let $(\X,\sfd)$ be a proper, geodesically complete, $\CAT(0)$ space and suppose it is $(L,C)$-quasi-isometric to $\R^n$. Then 
	\begin{itemize}
		\item[(i)] $\X$ admits the cone at infinity $C_\infty\X$ which is moreover $L$-biLipschitz equivalent to $\R^n$. Furthermore, $\partial_T\X$ is a compact $\CAT(1)$ homology $(n-1)$-manifold which is homotopy equivalent to $\mathbb{S}^{n-1}$.
		\item[(ii)] For every $x\in \X$ there exists $R_x > 0$ such that $S_R(x)$ is homotopy equivalent to the Tits boundary $\partial_T\X$,  hence to $\mathbb{S}^{n-1}$, for every $R\ge R_x$.
	\end{itemize}
\end{prop}

Notice that the  proposition above does not affirm that $\partial_T\X$ is a topological manifold, nor that $S_R(x)$ is a topological or even a homology manifold for big enough $R$, although it is a homology sphere; see footnote \ref{footnote:homology_sphere} and the discussion after Proposition \ref{prop:characterization_top_manifolds}. 

\begin{proof}
	Denote by $f\colon \X \to \R^n$ the $(L,C)$-quasi-isometry. Fix a point $x\in \X$ and a non-principal ultrafilter $\omega$. Consider the maps $f\colon (\X, \frac{1}{i}\sfd, x) \to (\R^n, \frac{1}{i} \sfd_E, f(x)) \cong (\R^n, \sfd_E, f(x))$, that are $(L, C/i)$-quasi-isometries, for $i\in \N$. This sequence of maps induce a well defined ultralimit map of pointed metric spaces $f_\omega \colon \omega\text{-}\lim(\X,\frac{1}{i}\sfd,x) \to \omega\text{-}\lim(\R^n,\frac{1}{i}\sfd_E,f(x)) \cong (\R^n,\sfd_E,0)$ which is $L$-biLipschitz. Therefore, $\omega\text{-}\lim(\X,\frac{1}{i}\sfd,x)$ is proper, so $C_\infty \X$ exists by Proposition \ref{prop:properties_asymptotic_cone} and coincides with $\textup{Cone}(\partial_T\X)$.
    Since $C_\infty\X$ is a topological manifold then $\partial_T\X$, which is isometric to the space of directions at the vertex of the Euclidean cone $C_\infty\X$, is a homology $(n-1)$-manifold which is homotopy equivalent to $\mathbb{S}^{n-1}$, by Propositions \ref{prop:characterization_hom_manifolds} and \ref{prop:characterization_top_manifolds}. The compactness of $\partial_T\X$ comes from Proposition \ref{prop:properties_asymptotic_cone}.
    Finally, (ii) is Proposition \ref{prop:homotopy_stability_large_spheres}.
\end{proof}

\begin{obs}
\label{rmk:2-dim-qi_implies_pure_dimensional}
{\em 
 One might   expect  that a $\CAT(0)$ space $(\X,\sfd)$ as in   Proposition \ref{prop:QI_to_R^n_has_biLip_asymptotic_cones} must be purely $n$-dimensional.   This is   true in dimension $2$.  Indeed, in this case,  the dimension of every point of $\X$ is at most $2$, so if there was a point $x\in \X$ of dimension $1$, then $\Sigma_x\X$ would be disconnected; 
     but this is impossible, since $\partial_T\X$ is connected and there is a $1$-Lipschitz surjective map $f\colon \partial_T\X \to \Sigma_x\X$.
     Therefore every point of $\X$ has dimension $2$.   \\ 
Clearly, this argument does not work for $n > 2$ and this seemingly innocent question is surprisingly hard.}
\end{obs}

The next example shows that, even for proper, geodesically complete, $\CAT(0)$ spaces $(\X,\sfd)$ admitting the cone at infinity, the connectedness of $\partial_T\X$ is not enough to guarantee pure-dimensionality in dimension bigger than $2$. Therefore, if the question in Remark \ref{rmk:2-dim-qi_implies_pure_dimensional} has positive answer, it has to be because of some other properties of the Tits boundary under the assumptions of Proposition \ref{prop:QI_to_R^n_has_biLip_asymptotic_cones}.

\begin{ex}
    There exists a proper, geodesically complete, $\CAT(0)$ space $(\X,\sfd)$ with $\dim(\X)=3$ satisfying the following properties:
    \begin{itemize}
       \item[(i)]  all spaces of directions $\Sigma_x X$  and all metric spheres $S_R(x)$ are connected;
          \item[(ii)] $(\X,\sfd)$ has the cone at infinity  which is purely $3$-dimensional, and $\partial_T\X$ is connected;
        \item[(iii)] $(\X,\sfd)$ is quasi-isometric to a proper, geodesically complete,   purely $3$-dimensional $\CAT(0)$ space (but not to $\R^3$);
         \item[(iv)] but $(\X,\sfd)$ is not purely $3$-dimensional.    
    \end{itemize}
    Let $S = \R \times [0,1]$. We construct $\X$ by gluing a copy of $\R^3$ to each boundary line of $S$ along a line in $\R^3$. The resulting space is $\CAT(0)$, geodesically complete, with $\dim(\X) = 3$ but it has points of dimension $2$. A direct check shows that every space of directions and every metric sphere is connected.    
    Moreover, $(\X,\sfd)$ is quasi-isometric to the gluing of two copies of $\R^3$ by a common line, which is a geodesically complete, purely $3$-dimensional $\CAT(0)$ space.  Since this latter space is a metric cone, it coincides with its cone at infinity, which is also the cone at infinity of $(\X,\sfd)$. The Tits boundary $\partial_T\X$ coincides with the section of this cone, namely the gluing of two $\mathbb{S}^2$ by two points, which is connected. 
    In particular,    $(\X,\sfd)$ is not quasi-isometric to $\R^3$, by Proposition \ref{prop:QI_to_R^n_has_biLip_asymptotic_cones}.(i).
\end{ex}

\vspace{2mm}
\section{More on homology manifolds}

We discuss here orientability and degree of maps between homology manifolds.\\ All the results stated in the first part are well-known for topological manifolds. Since it is hard to find a clean reference for the same results on homology manifolds, we try to give a unifying presentation. In the second part we will relate the topological notion of degree of a Lipschitz map to the analytical notion of its Jacobian by exploiting the rectifiability of $\CAT(1)$ homology manifolds. 

\subsection{Orientability and degree of maps}\label{sec:orientability}
${}$\\
The orientation sheaf of a  homology $n$-manifold $X$ is the sheaf ${\mathcal O}_X$ generated by the presheaf $U \mapsto H_n(U)$,   with stalks  ${\mathcal O}_{X,x} = \varinjlim_{U} H_n (U)$, which coincide with   the groups  $H_n(X, X \setminus \{x\} )$ see \cite[Corollary V.5.11]{bredon-book}.
A system of neighbourhoods for ${\mathcal O}_X$ is given by    the sets
$${\mathcal U} (U, \xi_U) = \{   j_x (\xi_U) \hspace{1mm} \,:\, \hspace{1mm} x \in U \}$$
for $U$ varying among all open subsets of $\X$,    $\xi_U \in H_n (\X, \X\setminus U)$, and where the homomorphism $j^U_x: H_n (\X, \X\setminus U) \to  H_n (\X, \X\setminus  \{ x \})$ is induced by the inclusion.

It is well-known that ${\mathcal O}_\X$ is a locally constant sheaf, that is  for every $x \in \X$ there  there exists $U$ such that  ${\mathcal O}_\X|_U \simeq U \times \Z$; this is usually subsumed by saying that  homology  manifolds are always {\em locally orientable}, see \cite{bredon-locor} and  \cite[Theorem 16.15]{bredon-book} 
. A connected homology manifold $\X$ is called {\em orientable} if the orientation sheaf is constant, that is $\mathcal O_\X \simeq \X \times \Z$;  
an {\em orientation} of $\X$ is a global section $\xi_\X \in \Gamma (\X,  \mathcal O_\X)$  corresponding to a generator of $\Z$.

The orientation sheaf of $\X$ over $\Z/2\Z$ is defined analogously, it is denoted by $\mathcal{O}_\X^{\Z/2\Z}$ and   is always constant, that is $\mathcal O_\X^{\Z/2\Z} \simeq  \X\times \Z/2\Z$; in this case,  if $\X$ is connected, the unique section $\xi_\X^{\Z/2\Z} \in \Gamma(\X,\mathcal{O}_\X^{\Z/2\Z})$ is called the {\em $\Z/2\Z$-orientation} of $\X$.

Homology $n$-manifolds satisfy the following form of Poincar\'e duality (see \cite{borel-poincare}, \cite[Theorem V.9.2]{bredon-book}): 
$$H_k(\X) \cong   H^{n-k} (\X, \mathcal O_\X),\text{ for }  0 \le k \le n$$ 
where the latter  are the  cohomology groups with coefficients in the orientation sheaf \footnote{Recall that the sheaf cohomology groups  $H^k(\X,{\mathcal F})$ coincide  with the \v{C}ech cohomology groups $\check H^k(\X,{\mathcal F})$ on any paracompact space $\X$, and with singular cohomology  for locally constant sheaves with finitely generated stalks (e.g. constant coefficients) when $\X$ is, moreover,  locally contractible (see \cite[Corollary 4.12]{bredon-book}).}.
In particular, since an open subset of a homology $n$-manifold is still a homology $n$-manifold, for any open subset $U \subset \X$, we have  
$$H_n(\X, \X \setminus U) \cong H_n(U) \cong  H^0 (U, \mathcal O_\X) = \Gamma (U, \mathcal O_\X)$$
(where the first isomorphism is obtained by excision, see  \cite[Corollary V.5.9]{bredon-book}, and $\Gamma (U, \mathcal O_\X)$ denotes the continuous sections of  $\mathcal{O}_\X$ over $U$).
As a consequence, a connected  homology $n$-manifold $\X$ is orientable if and only if 
$$   H_n(\X) \cong  H^0 (\X, \mathcal O_\X) = \Gamma (\X, \mathcal O_\X) =\Z .$$
In this case, the constancy of   $\mathcal O_\X$ implies that  
the  homomorphisms  induced by  inclusions
\begin{equation}\label{isomorfismi} 
H_n(\X) \stackrel{j_U}{\longrightarrow} H_n (\X, \X \setminus U) \stackrel{j^U_x}{\longrightarrow} H_n (\X, \X \setminus \{ x \}) \cong \Z
\end{equation}
  are   isomorphisms for every $x$ and every connected open neighbourhood $U$ of $x$. Then,  calling $j_x=  j^U_x \circ j_U$,  the element $j_x (\xi_\X)$ is a generator of  $H_n (\X, \X \setminus \{ x \})$ for every  $x$, denoted $\xi_x$.

For proper maps between oriented  homology $n$-manifolds there is a well-defined notion of {\em degree},  satisfying the usual properties as in the case of maps between  topological manifolds. Indeed, let $(\X, \xi_\X), (\Y, \xi_\Y)$ be oriented, connected, homology $n$-manifolds and let $f\colon \X \to \Y$ be a proper, continuous map: the degree of $f$ is the integer
 $$\textup{deg}(f) = f_\ast (\xi_\X) / \xi_\Y \,.$$ 
 The {\em local degree} of $f$ at  a point $x \in X$ which is an isolated point of $f^{-1}(y)$  is similarly defined as
 $\textup{deg}_x (f) := f_\ast (\xi_x) / \xi_y$, 
 where $f_\ast$ denotes  the homomorphism
 induced\footnote{The homomorphism $f_\ast$ always exists between Borel-Moore homologies with compact supports, and also with closed supports when $f$ is proper, cp. \cite[Proposition V.4.5]{bredon-book}}  between the relative homology groups $H_n(\X, \X \setminus \{x\}) \to H_n(\Y, \Y \setminus \{y\})$.

An analogous notion exists for   proper  maps between homology manifolds which are not necessarily orientable.
 By the Universal Coefficient Theorem, every homology $n$-manifold  satisfies $H_\ast(\X,\X\setminus\{x\},\Z/2\Z) \simeq H_\ast(\R^n,\R^n\setminus\{0\},\Z/2\Z) $ for every $x\in \X$.  
As before,   the inclusions yield  isomorphisms $j_U, j^U_x $ and $j_x= j^U_x \circ j_U$
$$ H_n(\X, \Z/2\Z ) \stackrel{j_U}{\longrightarrow} H_n (\X, \X \setminus U, \Z/2\Z ) \stackrel{j^U_x}{\longrightarrow} H_n (\X, \X \setminus \{ x \}, \Z/2\Z ) \cong \Z/2\Z $$
with $j_x$ sending the $\Z/2\Z$-orientation  $\xi_X^{\Z/2\Z}$ to the generator  $j_x^{\Z/2\Z}$ of  $H_n (\X, \X \setminus \{ x \}, \Z/2\Z)$.\\
Then, the {\em $\Z/2\Z$-degree}  of $f$  is  defined analogously by setting 
$$\deg_2(f)  = f_\ast (\xi_\X^{\Z/2\Z}) / \xi_\Y^{\Z/2\Z} \in   \Z/2\Z$$
while the   local degree at $x \in \X$  is defined as $\textup{deg}_{2,x} (f) := f_\ast (\xi_x^{\Z/2\Z}) / \xi_y^{\Z/2\Z}$, for $y=f(x)$.

\vspace{2mm}
 The following are classical properties of the degree we will need later.

\begin{lemma}
\label{lemma:degree}
Let $f\colon (\X, \xi_\X) \to (\Y, \xi_\Y)$ be a proper, continuous map between connected, oriented  homology $n$-manifolds. Assume that $f^{-1} (y) = \{ x_1,...,x_m \}$ consists of finitely many points. Then 
$$\textup{deg}(f) = \sum_{i=1}^m  \textup{deg}_{x_i} (f).$$ 
The same statement holds for any proper continuous map $f\colon \X \to \Y$ between connected homology $n$-manifolds with the $\Z/2\Z$-degree replacing the degree.
\end{lemma}

\begin{proof}
We only  give the proof in the oriented case, the formula for $\textup{deg}_2 (f)$ being analogous.
  Let $U_i \subseteq \X$ be connected open subsets such that $U_i \cap f^{-1}(y) = \{ x_i \}$. 
By excision we have  $H_n(\X, \X  \setminus \{ x_i \} ) \simeq H_n(U_i, U_i  \setminus \{ x_i \} )$ and by additivity $H_n(\X, \X \setminus f^{-1}(y)) \simeq \oplus_{i=1}^m H_n(\X, \X  \setminus  \{ x_i \} )$.
Then, the commutative diagram
\vspace{-4mm}

$$\xymatrix{
 \hspace{2mm} 
 \oplus_{i=1}^m H_n(\X, \X \setminus \{ x_i \}) \simeq \Z^m \hspace{2mm}  
   \ar[r]^{\hspace{5mm} f_\ast}   
  & \hspace{4mm} H_n(\Y, \Y \setminus \{ y \} )  \simeq   \Z 
\\
   \hspace{2mm} H_n(\X) \simeq \Z  \hspace{2mm}  
  \ar@{^{}->}[u]_{ \oplus_i j_{x_i}} 
  \ar[r]^{\hspace{5mm} f_\ast}     
  & \hspace{4mm} H_n(\Y)  \simeq \Z  
   \ar@{^{(}->}[u]_{j_y}    
 }
$$
implies that $ \deg (f) \cdot \xi_y = (j_y \circ f_\ast ) (\xi_X) = f_\ast (j_{x_1} (\xi_X), \ldots, j_{x_m} (\xi_X)) = \sum_{i=1}^m \deg_{x_i} (f) \xi_y.$
\end{proof}
  
\begin{obs}
\label{rmk:degree_covering}
{\em  If $f\colon (\X, \xi_\X) \to (\Y, \xi_\Y)$ is a $d$-sheeted covering   between  connected, oriented  homology $n$-manifolds, then $|\textup{deg}(f)| =  d $. 
Indeed,  if $U \subset \X$ is a connected open subset such that $f|_U\colon U \to V=f(U)$ is a homeomorphism, then for every $x \in U$ and $y=f(x)$ we have,  by \eqref{isomorfismi},
$$f_\ast (\xi_x) = f_\ast ( (j^U_x \circ j_U) (\xi_\X)) 
 = (j^V_y \circ j_V) (f_\ast (\xi_\X)) 
 =j^V_y ( \deg (f|_U) \cdot \xi_\Y ) = \deg (f|_U) \cdot \xi_y$$
 which shows that $\deg_x (f) = \deg (f|_U) $ is locally constant on $\X$.  Hence,  $\deg_x (f)$ is constant by connectedness of $\X$,  everywhere equal to $1$ or $-1$ (as $f|_U$ is a a homeomorphism). By the degree formula of Lemma \ref{lemma:degree} it then follows that $\deg (f)= \pm d$. Similarly, if $f\colon \X \to \Y$ is a $d$-sheeted covering between connected homology $n$-manifolds, then
 $\textup{deg}_2(f) \equiv d$ in $\Z/2\Z$.}
\end{obs}

\begin{lemma}
\label{lemma:degree_is_multiplicative}
    Let $f\colon (\X,\xi_\X) \to (\Y,\xi_\Y)$ and $g\colon (\Y,\xi_\Y) \to ({\rm Z}, \xi_{\rm Z})$ be proper, continuous maps between connected, oriented homology $n$-manifolds: then,  $\textup{deg}(g\circ f) = \textup{deg}(g)\cdot \textup{deg}(f)$. \\
    A similar statement holds for the $\Z/2\Z$-degree.
\end{lemma}
\begin{proof}
    By definition, $\textup{deg}(g\circ f) \cdot\xi_\X = \xi_{\rm Z} = \textup{deg}(g)\cdot \xi_\Y = \textup{deg}(g)\cdot \textup{deg}(f)\cdot \xi_\X$.
\end{proof}

The $\Z/2\Z$-degree of any proper map from an orientable manifold to a non-orientable one is zero.  

\begin{lemma}
\label{lemma:map_orientable_to_non_orientable_0_degree}
    Let $f\colon \X  \to \Y$ be a proper, continuous map between   connected, homology $n$-manifolds. Assume $\X $ is orientable and $\Y$ is not: then, $\textup{deg}_2(f) = 0$.
\end{lemma}
\begin{proof}
By \cite[Sec. V.3, eq.(13)]{bredon-book},
we have  exact sequences
$$\xymatrix{
    H_n(\X,\Z)\otimes \Z/2\Z 
    \ar[r]^{\hspace{3mm} r_X}  \ar@<+2ex>@{->}[d]_{f_\ast}    
  &  H_n(\X,\Z/2\Z) 
    \ar[r]^{} 
     \ar@{^{}->}[d]_{f_\ast}   
  &      0  
\\
        H_n(\Y,\Z)\otimes\Z/2\Z  
    \ar[r]^{\hspace{3mm}r_Y}               
  &  H_n(\Y,\Z/2\Z) 
    \ar[r]^{}
   & 0  
}
$$
where  $H_n(\Y,\Z)=0$ since  $\Y$ is non-orientable.
As $r_X$ sends any orientation $\xi_{\X}$ of $\X$  to the generator of $H_n(\X,\Z/2\Z) \simeq \Z/2\Z$, we deduce that 
$f_\ast  (\xi_X^{\Z/2\Z}) = (r_Y \circ f_\ast) (\xi_X ) =0$.
\end{proof}

\subsection{Degree, Jacobian and Area Formula}${}$\\
Let $n\in \N$. A metric space $(\X,\sfd)$ is $n$-rectifiable if there exist countably many biLipschitz maps $f_i\colon A_i\subseteq \R^n \to \X$ such that $\mathcal{H}^n(\X\setminus\bigcup_{i\in \N}f_i(A_i)) = 0$. The next statement presents the only example of rectifiable spaces which is relevant for us.
\begin{prop}[{\cite[Theorem 1.2]{LN19}}]
\label{prop:CAT_are_rectifiable}
    Let $(\X,\sfd)$ be a proper, geodesically complete, purely $n$-dimensional, $\CAT(1)$ space. Then there exists a subset $M\subseteq \X$ which is locally biLipschitz equivalent to $\R^n$ and such that $\mathcal{H}^n(\X\setminus M) = 0$. In particular, $(\X,\sfd)$ is $n$-rectifiable.
\end{prop}

The theory of Jacobian of Lipschitz maps on rectifiable spaces is developed in \cite{AmbrosioKirchheim2000}. Firstly, they consider two metric spaces $(\X,\sfd_\X)$, $(\Y,\sfd_\Y)$ that are subsets of duals of separable Banach spaces $V^\ast$ and $W^\ast$ respectively. 
Namely, let $V, W$ be   separable Banach spaces, let $(\X,\sfd_\X) \subseteq (V^*, \Vert \cdot \Vert_{V^*})$ be a $n$-rectifiable metric space and let 
$f: (\X,\sfd_\X) \to (W^\ast, \Vert \cdot \Vert_{W^\ast})$
be a Lipschitz map. Then, for $\mathcal{H}^n$-a.e. $x\in \X$ there exists a suitable $n$-dimensional vector subspace of $V^*$ called the {\em approximate tangent space} at $x$ and denoted by $\textup{Tan}(x,\X)$ (see  \cite[\S 5]{AmbrosioKirchheim2000} for the properties of this space). The dual $W^*$   of a separable Banach space $W$ carries the norm metric $\Vert\cdot\Vert_{W^*}$ and also, for any choice of a dense subset $\{w_i\}_{i\in \N}$  of $W$,  the metric
$$\sfd_{W^*}(T,T') := \sum_{i=0}^\infty 2^{-i}\langle T-T', w_i\rangle,$$
inducing the weak-$\ast$ topology. 
Then, \cite[Theorem 8.1]{AmbrosioKirchheim2000} gives that for $\mathcal{H}^n$-a.e. $x\in \X$ there exist a linear and weak-$\ast$ continuous map $A\colon V^* \to W^*$ and a Borel subset $S_x\subseteq \X$ such that 
\begin{equation}
    \label{eq:differential_Lip_map_rectifiable}
    \lim_{y \in \X\setminus S_x\to x} \frac{\sfd_{W^*}(f(y), f(x) + A(y-x))}{\Vert y-x \Vert_{V^*}} = 0,
\end{equation}
where $S_x$ has the property that for every $v \in \textup{Tan}(x,\X) \cap \{\Vert \cdot \Vert_{V^*} =1\}$ there exists a sequence $\{y_i\} \in \X  \setminus  S_x$ such that $y_i\to x$ and $\frac{y_i-x}{\Vert y_i - x \Vert}$ converges weakly-$\ast$ to  $v$  \cite[Proposition 5.7]{AmbrosioKirchheim2000}.\\
The restriction of $A$ to $\textup{Tan}(x,\X)$ is uniquely determined and it is denoted by $${\mathbf{d}}_xf \colon \textup{Tan}(x,\X) \to W^* \;.$$ The Jacobian of $f$ at $x$ is defined as the number
$$\mathbf{J}f(x) := \frac{\omega_n}{\mathcal{H}^n(\{v\in \textup{Tan}(x,\X)\,:\, \Vert {\mathbf{d}}_xf(v) \Vert_{W^*}\le 1\})},$$
where $\omega_n$ denotes the $\mathcal{H}^n$-measure   of the unit ball $\mathbf{B}_{1}$  in $(\textup{Tan}(x,\X), \Vert \cdot \Vert_{V^*})$
(which coincides with the volume of the unit ball
in the Euclidean space $\R^n$  by \cite[Lemma 6]{Kirchheim1994}).

\begin{obs}{\em
\label{remark:Jacobian_Lipschitz}
    In the situation above, if $f$ is $L$-Lipschitz then $0\le \mathbf{J}f(x) \le L^n$. Indeed, using the definition of $\sfd_{W^*}$ and \eqref{eq:differential_Lip_map_rectifiable} we get that
    \begin{equation*}
        \begin{aligned}
            \lim_{y \in \X\setminus S_x\to x} \sum_{i=0}^\infty 2^{-i}\left\langle A\left(\frac{y-x}{\Vert y - x \Vert_{V^*}}\right), w_i \right\rangle &= \lim_{y \in \X\setminus S_x\to x} \sum_{i=0}^\infty 2^{-i}\left\langle \frac{f(y)-f(x)}{\Vert y - x \Vert_{V^*}}, w_i \right\rangle \\
            &\le  \lim_{y \in \X\setminus S_x\to x} \sum_{i=0}^\infty 2^{-i}L\left\langle \frac{f(y)-f(x)}{\Vert f(y) - f(x) \Vert_{W^*}}, w_i \right\rangle \\
            &\le  \lim_{y \in \X\setminus S_x\to x} \sum_{i=0}^\infty 2^{-i}L \Vert w_i\Vert_{W}
        \end{aligned}
    \end{equation*}
Choosing $v \in  \mathbf{B}_{1}$ 
such that $\Vert A(v) \Vert_{W^*} = \Vert {\mathbf{d}}_xf \Vert$, there exist
points $y \in \X\setminus S_x \to x$ such that $\Vert A(\frac{y-x}{\Vert y - x \Vert_{V^*}} )\Vert$ converges in the weak-$\ast$ sense to $A(v)$, the property of the set $S_x$ recalled above.
Applying the weak-$\ast$ convergence we get
    $$\sum_{i=0}^\infty 2^{-i}\Vert {\mathbf{d}}_xf \Vert\Vert w_i\Vert_{W} \le \sum_{i=0}^\infty 2^{-i}L \Vert w_i\Vert_{W},$$
    which implies that $\Vert {\mathbf{d}}_xf \Vert \le L$. 
 Then, the $\Vert \cdot \Vert_{V^*}$-ball  of radius  $\frac{1}{L}$ in $\textup{Tan}(x,\X)$ satisfies
 $$\mathbf{B}_{1/L} \subseteq\{v\in \textup{Tan}(x,\X)\,:\, \Vert {\mathbf{d}}_xf(v) \Vert_{W^*}\le 1\},$$
    so 
    $$\mathcal{H}^n(\{v\in \textup{Tan}(x,\X)\,:\, \Vert df_x(v) \Vert_{W^*}\le 1\}) \ge \mathcal{H}^n( \mathbf{B}_{1/L}) = \frac{1}{L^n}\omega_n.$$
   Therefore, the definition of Jacobian gives $0\le \mathbf{J}f(x)\le L^n$.
    }
\end{obs}

 To define  the Jacobian in the general case of metric spaces which are not subsets of Banach spaces, we proceed as follows. Let $f\colon(\X,\sfd_\X) \to(\Y,\sfd_\Y)$ be a Lipschitz map between separable metric spaces and assume that $(\X,\sfd)$ is $n$-rectifiable. Consider the Kuratowski embeddings $\iota_\X\colon \X \to \ell^\infty$ and $\iota_\Y\colon \Y \to \ell^\infty$, and the induced map $\tilde{f}:=\iota_\Y\circ f\circ\iota_\X^{-1}\colon \iota_\X(\X) \to \iota_\Y(\Y)\subseteq \ell^\infty$. 
 Then, we define $\mathbf{J}f(x) := \mathbf{J}\tilde{f}(\iota_\X(x))$.

Even if the definition of Jacobian is extrinsic, it  has an intrinsic nature that is highlighted by the area-formula.

\begin{prop}[{\cite[Theorem 8.2]{AmbrosioKirchheim2000}}]
\label{prop:area_formula}
    Let $f\colon(\X,\sfd_\X) \to (\Y,\sfd_\Y)$ be an $L$-Lipschitz map and assume that $(\X,\sfd_\X)$ is $n$-rectifiable. Then
    $$\int_\X g(x)\, \mathbf{J}f(x)\,\d\mathcal{H}^n(x) = \int_\Y \sum_{x\in f^{-1}(y)}g(x)\,\d\mathcal{H}^{n}(y)$$
    for every Borel function $g\colon \X \to \R$. In particular, $\int_\X \mathbf{J}f(x)\,\d\mathcal{H}^n(x) = \int_\Y \#f^{-1}(y)\,\d\mathcal{H}^{n}(y)$. \\Moreover, $0\le \mathbf{J}f(x) \le L^n$.
\end{prop}
\begin{proof}
    The first statement is {\cite[Theorem 8.2]{AmbrosioKirchheim2000}},  the second one follows by Remark \ref{remark:Jacobian_Lipschitz}.
\end{proof}

Clearly, for any map $f$ between open subsets of $\R^n$ which is differentiable at a point $x$, the  differential $\mathbf{d}_x f$ and the Jacobian $\mathbf{J}f(x)$ defined above coincide with the usual notions in $\R^n$. In particular,  the following lemma holds.

\begin{lemma}\label{lemmagradoloc} Let $f:U \subset \R^n \to  V\subset \R^n$ be a Lipschitz map between open sets, and let $x$ be a point  
where $f$ is differentiable such that  $\mathbf{J}f (x)> 0$. Then, $x$ is an isolated point of $f^{-1}(y)$ and $|\textup{deg}_x(f)|=1$.
The same is true for the  $\Z/2\Z$-degree.
\end{lemma}

\begin{proof}
Since $x$ is a point where $f$ is differentiable, the determinant of $df_x$ coincides with the Jacobian $\mathbf{J}f(x)$; moreover, since $\mathbf{J}f(x)> 0$, this implies that there exists $\lambda>0$ such that 
\begin{equation}\label{eqexpanding}
\Vert df_x (v) \Vert \geq \lambda \Vert v \Vert \hspace{5mm}\mbox{ for every } v \in \R^n.
\end{equation}
  In particular, $\frac{1}{\lambda}f$ is expanding in a neighbourhood of $x$,
hence $f$ is proper on $B_r(x)-\{ x \}$ for $r$ small enough and $x$ is the only preimage of $y$ in  $B_r(x)$. Let now $g(x'):=x+(df)_x (x'-x)$, let $V_y$ be a convex neighbourhood of $y$,   and choose $0<\varepsilon < \lambda$ and a possibly smaller $r>0$ such that $B_r(x) \cap f^{-1}(y)=\{ x\}$, 
$f(B_r(x))\cup g(B_r(x))  \subset V_y$ and
$\Vert f(x')-g(x') \Vert <\varepsilon \Vert x'-x \Vert$ for all $x' \in B_r(x)$.
Then, the maps $f$ and $g$ are properly homotopic as maps of pairs $\left(B_r(x),B_r(x) \setminus \{x\} \right) \to \left(V_y,V_y \setminus \{y \} \right)$.
Actually,  calling $h_t= tf+(1-t)g$,  we have that
$h_t(x') \neq y$ for all $x' \in B_r(x) \setminus \{ x\}$ 
(in fact,  setting $\Delta(x') = f(x')-g(x')$,   by the differentiability of $f$ at $x$ we have 
$\sfd(h_t(x'), y)=\Vert (d_x f)(x'-x)  + t  \Delta(x') \Vert \geq (\lambda - t \varepsilon) \Vert x' -x\Vert  >0$ for all $x' \in B_r(x) \setminus \{ x\}$, by \eqref{eqexpanding}).
It follows that $f$ and $g$ induce the same map   in relative homology   $H_n \left(U,U\setminus \{ x\}\right)  = H_n \left(B_r(x),B_r(x) \setminus \{ x\}\right) \to H_n\left(V_y, V_y \setminus \{y \}\right) =   H_n \left(V,V \setminus \{ y \} \right)$ (with coefficients $\Z$ or $\Z/2\Z$), hence $\textup{deg}_x (f) = \textup{deg}_x (g)$  and $\textup{deg}_2 (f) = \textup{deg}_2 (g)$ as well. But the degree of $g$  is clearly $\pm 1$ as $g$  is a linear isomorphism.
\end{proof}

The next one is a generalization of classical results to $\CAT(1)$ homology manifolds.

\begin{prop}
\label{prop:degree_area_formula_CAT}
    Let $f\colon (\X,\sfd_\X) \to (\Y,\sfd_\Y)$ be a Lipschitz map between two compact, connected, $\CAT(1)$, homology $n$-manifolds. 
    Then:
    \begin{itemize}
        \item[(i)] $\textup{deg}_2(f) \equiv \#f^{-1}(y)$ in $ \Z/2\Z$, for $\mathcal{H}^n$-a.e. $y\in \Y$;
        \item[(ii)] if $\X$ and $\Y$ are orientable then $\vert\textup{deg}(f)\vert \le \#f^{-1}(y)$ for $\mathcal{H}^n$-a.e. $y\in \Y$. 
    \end{itemize}
\end{prop}
\begin{proof}
    By Lemma \ref{lemma:CAT_homology_are_purely_dimensional}, the spaces $\X, \Y$ are purely $n$-dimensional and geodesically complete.
    Let $M_\X\subseteq \X$, $M_\Y\subseteq \Y$ be the subsets locally biLipschitz to open subsets of $\R^n$ provided by Proposition \ref{prop:CAT_are_rectifiable}. Since $f$ is Lipschitz we have $\mathcal{H}^n(\X\setminus M_\X) = 0$. The Jacobian $\mathbf{J}f(x)$ is a Borel  function $\X\to [0,+\infty)$ defined $\mathcal{H}^n$-a.e.. We apply Proposition \ref{prop:area_formula} with $g=\chi_A$ to obtain $\mathcal{H}^n(f(A)) = 0$. Therefore, for $\mathcal{H}^n$-a.e. $y$  in $\Y$ we can suppose that $y \in M_\Y$, every $x\in f^{-1}(y)$ belongs to $M_\X$, satisfies $\mathbf{J}f(x) \neq 0$ and is a differentiable point for $f$, by Rademacher's Theorem. 
    By taking biLipschitz charts for $M_\X$ and $M_\Y$ around the points $x$ and $y$ respectively, we may assume that $f:U \to V$ is a Lipschitz map between connected open subsets of $\R^n$, and that $f$ is differentiable at $x$ with   $\mathbf{J}f(x) \neq 0$.
    Moreover, by Lemma \ref{lemmagradoloc} we know that $x$ is an isolated point of  $f^{-1}(y)$ and that $ \vert \textup{deg}_{2,x}(f)\vert = 1$.
   Combining this property with Lemma \ref{lemma:degree} we get that   for $\mathcal{H}^n$-a.e. $y\in \Y$ we have
    $$\textup{deg}_2(f) = \sum_{x_i\in f^{-1}(y)} \textup{deg}_{2,x_i}(f) = \#f^{-1}(y) \textup{ mod } \Z/2\Z.$$
    If $\X,\Y$ are orientable, we deduce as before that for $\mathcal{H}^n$-a.e. $y\in \Y$ we have  $\vert \textup{deg}_{x}(f)\vert = 1$ for all $x \in f^{-1}(y)$,  hence for such $y$'s it holds 
    $$\vert\textup{deg}(f)\vert = \left\vert \sum_{x_i\in f^{-1}(y)}\textup{deg}_{x_i}(f) \right\vert \le \sum_{x_i\in f^{-1}(y)}\vert \textup{deg}_{x_i}(f) \vert \le  \#f^{-1}(y).$$ 
\end{proof}

A straightforward combination of Propositions \ref{prop:area_formula} and \ref{prop:degree_area_formula_CAT} gives the next result.

\begin{cor}
\label{cor:degree_volume_inequality}
     Let $f\colon (\X,\sfd_\X) \to (\Y,\sfd_\Y)$ be a $1$-Lipschitz map between two compact, connected, $\CAT(1)$, oriented,
$n$-homology manifolds. Then
     $$\mathcal{H}^n(\X) \ge \vert \textup{deg}(f)\vert \cdot \mathcal{H}^n(\Y).$$
\end{cor}



\vspace{2mm}
\section{The proof of Theorems  \ref{theo:intro_qi_cocompact} and \ref{theo:intro_qi_implies_biLip}}
\label{sec:Theorems_Nagano_and_cocompact}

We recall the statement of Theorem \ref{theo:intro_qi_cocompact}, in the cocompact case.

\begin{T2}
    Let $(\X,\sfd)$ be a proper, geodesically complete, $\CAT(0)$ space such that $\Isom(\X,\sfd)$ acts cocompactly on $\X$. If $(\X,\sfd)$ is quasi-isometric to $\R^n$, then $(\X,\sfd)$ is isometric to $\R^n$.
\end{T2}

\begin{proof}
	Since $(\X,\sfd)$ is quasi-isometric to $\R^n$ then its Tits boundary $\partial_T\X$ is compact, by Proposition \ref{prop:QI_to_R^n_has_biLip_asymptotic_cones}. Therefore, \cite[Proposition 7]{Bosche2011} gives the thesis.
\end{proof}

Since \cite{Bosche2011} is not published, we propose an alternative proof.

\begin{proof}[Alternative proof of Theorem \ref{theo:intro_qi_cocompact}]
	Set $G:=\Isom(\X,\sfd)$ and fix a non-principal ultrafilter $\omega$. Consider the group $G_\omega$    acting by isometries on the geodesically complete, $\CAT(0)$ space $C_\omega \X$.\\ This action is transitive, being the limit of cocompact actions with codiameter going to $0$. By Proposition \ref{prop:QI_to_R^n_has_biLip_asymptotic_cones}, $C_\omega \X$ is proper and it is a topological manifold.
	Therefore, by \cite[Theorem 3]{Berestovskii1989}, $C_\omega \X$ is a smooth Lie group whose metric coincides with a sub-Finsler Carnot-Caratheodory metric. Since the 
    metric is $\CAT(0)$, then it is indeed Finsler.
    This is due to the fact that the Hausdorff dimension of a sub-Finsler Carnot-Caratheodory metric that is not Finsler is strictly larger than the topological dimension of the Lie group, while the Hausdorff dimension of a geodesically complete, $\CAT(0)$ metric coincides with the topological dimension, by \cite[Theorem 1.1]{LN19}. A Finsler manifold which is $\CAT(0)$ is indeed Riemannian by \cite[Theorem 1.2]{BuckleyFalkWraith2009}. Hence, since $C_\omega \X$ is biLipschitz to $\R^n$, we deduce that $C_\omega \X$ has to be an abelian Lie group by \cite[Theorem A]{Pauls2001}. Since it is $\CAT(0)$, hence simply connected, $C_\omega \X$ is isometric to $\R^n$. This implies that the Tits boundary $\partial_T\X$ is isometric to $\mathbb{S}^{n-1}$. \\
    From this, one immediately infers that   $\X$ is flat,  by induction on $n$, as follows. \\
If $n=1$, then $\partial_T\X \cong \mathbb{S}^0$ and  one gets that $\X = \R$ by geodesic completeness.\\
If $n>1$  we have that $\mathbb{S}^{n-1} \cong \mathbb{S}^0 \ast \mathbb{S}^{n-2}$, where $\ast$ denotes the spherical join (see comment after Lemma \ref{lemma:cones_and_joins}); then, by Lemma \ref{lemma:cones_and_joins},  $\X = \X_1\times \X_2$ where each factor is again a proper, geodesically complete  $\CAT(0)$ space,  and  $\partial_T\X_1 = \mathbb{S}^0$ and $\partial_T\X_2 = \mathbb{S}^{n-2}$. Then,  we deduce  by induction   that $\X_1 \cong \R$ and $\X_2 \cong \R^{n-1}$, hence $\X \cong \R \times \R^{n-1} \cong \R^n$.
\end{proof}

We remark that if one assumes that $(\X,\sfd)$ has a cocompact group of isometries without fixed points at infinity then the conclusion follows directly from \cite{AB97}.

\begin{prop}
	Let $(\X,\sfd)$ be a proper, geodesically complete, $\CAT(0)$ space. Suppose that there exists a group of isometries $G < \Isom(\X,\sfd)$ which is cocompact and does not fix points on $\partial_T\X$. If $\partial_T\X$ is compact, e.g. if $(\X,\sfd)$ is quasi-isometric to some $\R^n$, then $\X$ is flat.
\end{prop}
\begin{proof}
	Since $\partial_T\X$ is compact, then $\X$ is metrically doubling by Proposition \ref{prop:properties_asymptotic_cone}. This implies that $G$ has polynomial growth, so it is amenable. Since $G$ has no fixed points at infinity then \cite[Theorem A]{AB97} implies that $\X$ is flat.
\end{proof}

Let us now turn to the proof of Theorem \ref{theo:intro_qi_implies_biLip}.

\begin{T1}
    Given $\varepsilon \! > \! 0$ and $n \in \N$, there exists $\delta \! = \!\delta(\varepsilon, n) \!>\! 0$ such that the following holds.
	Let $(\X,\sfd)$ be a proper, geodesically complete, $\CAT(0)$ space.
    \begin{itemize}[leftmargin=7mm]
    \item[(i)] If $(\X,\sfd)$ is $(1+\delta, C)$-quasi-isometric to $\R^n$, then  $(\X,\sfd)$
      is $(1+\varepsilon)$-biLipschitz homeomorphic to $\R^n$.
    \item[(ii)] Moreover, if $(\X,\sfd)$ is $(1, C)$-quasi-isometric to $\R^n$, then $(\X,\sfd)$ is isometric to $\R^n$.
    \end{itemize}
\end{T1}

In order to prove it,
we recall the definition of the $n$-asymptotic volume of a proper, geodesically complete, $\CAT(0)$ space $(\X,\sfd)$ introduced in \cite{Nagano2022}. It is the quantity
$$\textup{as-vol}_n(\X) := \limsup_{R\to +\infty} \frac{\mathcal{H}^n(B_R(x))}{\omega_n R^n} \in [0,+\infty],$$ 
where $\omega_n$ is the $n$-dimensional volume of the ball of radius $1$ in $\R^n$ and $x$ is a point of $\X$ (notice that $\textup{as-vol}_n(\X)$ does not depend on the choice of the point $x$, by the triangular inequality).
Nagano proved in \cite[Proposition 3.5]{Nagano2022} that, if  the asymptotic cone $C_\infty \X$ exists, then  
\begin{equation}
\label{eqasvol-nagano}
\textup{as-vol}_n(\X) = \frac{\mathcal{H}^{n-1}(\partial_T\X)}{\mathcal{H}^{n-1}(\mathbb{S}^{n-1})} \end{equation}
and that if   $\X$ is purely $n$-dimensional and $\textup{as-vol}_n(\X)$ is close to $1$, then $(\X,\sfd)$ is bi-Lipschitz equivalent to $\R^n$. 
Actually, looking at his proof, it is enough to assume that $\partial_T\X$ is compact and purely $(n-1)$-dimensional.
So, we state his result in this form.

\begin{theo}[{\cite[Theorem 1.1]{Nagano2022}}]
	\label{theo:Nagano_biLip}
	For every $\varepsilon > 0$ and   $n\in \N$ there exists $\delta_0 \!=\! \delta_0(\varepsilon,n) > 0$ with the following property. If $(\X,\sfd)$ is a proper, geodesically complete, $\CAT(0)$ space with $\textup{as-vol}_n(\X) \le 1+\delta_0$ and such that $\partial_T\X$ is compact and purely $(n-1)$-dimensional, then $(\X,\sfd)$ is $(1+\varepsilon)$-biLipschitz equivalent to $\R^n$.
\end{theo}
Assuming that $\X$ is purely $n$-dimensional, Nagano proved that one can relax the  asymptotic volume bound and still obtain that  $\X$ is  homeomorphic to $\R^n$.

\begin{theo}[{\cite[Theorem 1.2]{Nagano2022}}]
	\label{theo:Nagano_3/2}
	If $(\X,\sfd)$ is a proper, geodesically complete, purely $n$-dimensional, $\CAT(0)$ space with $\textup{as-vol}_n(\X) < 3/2$, then $(\X,\sfd)$ is homeomorphic to $\R^n$.
\end{theo}

The constant $3/2$ is sharp, as shown in \cite[Theorem 1.3]{Nagano2022}. \\
 Assuming further   that $\X$ is a homology $n$-manifold, then the conclusion of Theorem \ref{theo:Nagano_3/2} remains true if $\textup{as-vol}_n(\X) < 3/2 + \varepsilon_n$, where $\varepsilon_n$ is a small positive constant (cp. \cite[Theorem 1.4]{Nagano2022}). 
 This is the best result one can obtain using the methods of \cite{Nagano2022}. In particular, Theorems \ref{theo:intro_homology_qi_R^n} and \ref{theo:intro_2} are out of reach, and that is why we follow a different strategy in Sections \ref{sec-DE} and \ref{sec:optimality_constant_2}.\\
Here, we focus on the proof of Theorem \ref{theo:intro_qi_implies_biLip}, which follow directly from Theorem \ref{theo:Nagano_biLip}, Proposition \ref{prop:QI_to_R^n_has_biLip_asymptotic_cones} and the following observation. 
\begin{lemma}
\label{lemma:as_vol_cone_infinity}
    Let $(\X,\sfd)$ be a proper, geodesically complete, $\CAT(0)$ space admitting the cone at infinity. Then $\textup{as-vol}_n(\X) = \textup{as-vol}_n(C_\infty\X)$.
\end{lemma}
\begin{proof}
    By Proposition \ref{prop:properties_asymptotic_cone}, we have $C_\infty\X = \textup{Cone}(\partial_T\X)$ . Let us denote by $o_\infty$ the vertex of the cone at infinity. The cone structure, together with \cite[Lemma 6.4]{Na02}, gives that $\mathcal{H}^n(B_R(o_\infty)) = \frac{R^n}{n}\cdot\mathcal{H}^{n-1}(\partial_T\X)$ for every $R>0$. The same formula applied to $\R^n$ gives $\omega_nR^n = \mathcal{H}^n(B_R(0)) = \frac{R^n}{n}\cdot\mathcal{H}^{n-1}(\mathbb{S}^{n-1})$, where $0\in \R^n$. Therefore, using \eqref{eqasvol-nagano} we have
    $$\textup{as-vol}_n(C_\infty\X) = \limsup_{R\to +\infty}\frac{\mathcal{H}^n(B_R(o_\infty))}{\omega_n R^n} = \limsup_{R\to +\infty}\frac{\mathcal{H}^n(B_R(o_\infty))}{\mathcal{H}^n(B_R(0))} = \frac{\mathcal{H}^{n-1}(\partial_T\X)}{\mathcal{H}^{n-1}(\mathbb{S}^{n-1})} = \textup{as-vol}_n(\X).$$
    \qedhere
\end{proof}

\begin{proof}[Proof of Theorem \ref{theo:intro_qi_implies_biLip}]
Let $\varepsilon > 0$,  let $\delta_0 = \delta_0(\varepsilon,n) > 0$ be the costant given by Theorem \ref{theo:Nagano_biLip}, and set $\delta := (1+\delta_0)^{\frac{1}{n}} - 1 >0$. 
Let $(\X,\sfd)$ be a proper, geodesically complete, $\CAT(0)$ space which is $(1 + \delta, C)$-quasi-isometric to $\R^n$. Proposition \ref{prop:QI_to_R^n_has_biLip_asymptotic_cones} implies that $C_\infty \X$ exists and it is $(1+\delta)$-biLipschitz equivalent to $\R^n$, and that $\partial_T\X$ is compact and purely $(n-1)$-dimensional. Then, by Lemma \ref{lemma:as_vol_cone_infinity}, we have
	\begin{equation}
        \textup{as-vol}_n(\X) = \textup{as-vol}_n(C_\infty\X) \le (1+\delta)^{n} \le 1+\delta_0.
	\end{equation}
Therefore,  Theorem \ref{theo:Nagano_biLip} implies that $(\X,\sfd)$ is $(1+\varepsilon)$-biLipschitz equivalent to $\R^n$.\\
Finally, if  $(\X,\sfd)$ is $(1, C)$-quasi-isometric to $\R^n$,   we can apply (i) to $\varepsilon_k = 1/k$ for $k\in \N$, and find a $(1+\varepsilon_k)$-biLipschitz homeomorphism $f_k\colon \X \to \R^n$. After composing these $f_k$  with  suitable translations of $\R^n$, we can suppose that all of them send a fixed point $x\in \X$ into the origin of $\R^n$. The sequence $\{f_k\}$ is then equi-Lipschitz and equi-bounded, so by Ascoli-Arzelà we can extract a converging subsequence. The limit map is therefore the desired isometry. \\
(Remark: assertion (ii) can be deduced  directly, without invoking Theorem \ref{theo:Nagano_biLip}, since    if $(\X,\sfd)$  is $(1,C)$-quasi-isometric to $\R^n$,  then $\partial_T\X$ is isometric to $\mathbb{S}^{n-1}$ by Proposition \ref{prop:QI_to_R^n_has_biLip_asymptotic_cones}, which easily implies that    $\X$ is flat, as explained at the end of the proof of Theorem \ref{theo:intro_qi_cocompact}.)
\end{proof}


\vspace{2mm}
\section{The proof of Theorem \ref{theo:intro_homology_qi_R^n}}
\label{sec:thm_homology_qi_R^n}

In this section we prove Theorem \ref{theo:intro_homology_qi_R^n} by developing the main tools needed for the study of asymptotic properties of $\CAT(0)$ homology manifolds.  
We recall the statement of the theorem.

\begin{T8}
    Let $(\X,\sfd)$ be a proper, $\CAT(0)$ homology $m$-manifold. If $(\X,\sfd)$ is quasi-isometric to $\R^n$, then $\X$ is homeomorphic to $\R^n$.
\end{T8}

The proof is mainly inspired by the works of Thurston \cite{Thurston1996} and Davis-Januszkiewicz \cite{DavisJanuszkiewicz1991}. We will actually extend some of the results of \cite{DavisJanuszkiewicz1991} beyond the PL-case.
The two main tools are provided by the following classical results.

\begin{prop}[{\cite[Corollary 2.10]{Thurston1996}}]
	\label{prop:Thurston}
	Let $(\X,\sfd)$ be a proper, $\CAT(0)$ homology $n$-manifold. Let $x\in \X$ and $0<r\le R$. Then:
	\begin{itemize}
		\item[(i)] the metric sphere $S_r(x)$ is a homology $(n-1)$-manifold;
		\item[(ii)] the natural contraction map $\varphi_x^{R,r} \colon S_R(x) \to S_r(x)$ is acyclic, that is for every $y \in S_r(x)$,  calling   $F_y=  (\varphi_x^{R,r})  ^{-1}(y)$   the fiber of $y$,  we have  $\tilde{H}^*(  F_y; \Z) = 0$, where $\tilde{H}^*$ is the reduced cohomology. 
	\end{itemize}
\end{prop}

\begin{obs}{\em 
The proof of the above result, in \cite{Thurston1996}, is based on \cite[Proposition 2.8]{Thurston1996}  which uses  singular homology, and  is a form of Alexander duality.
Proposition \ref{prop:Thurston}
still holds  if one considers  Borel-Moore homology in place of the singular homology considered there. \\ 
Indeed,  following the proof of \cite[Proposition 2.8]{Thurston1996}, we have that the closed balls $\overline{B}_r(x)$ are homology $n$-manifolds with boundary, and their boundary are the metric spheres $S_r(x)$, which are indeed   homology $(n-1)$-manifolds (defined via  Borel-Moore homology)  by \cite{Mitchell1990}. 
However, the sphere $S_r(x)$ might be not locally contractible (see Remark \ref{rmk:BorelMoore_vs_Singular}),  so the singular and Borel-Moore homologies may differ. 
The double of $\overline{B}_r(x)$ along its boundary is, in any case, a (Borel-Moore) homology $n$-manifold without boundary and a (Borel-Moore) homology sphere. This   follows from references [48] and [58] in \cite{Thurston1996}, where these results are stated for cohomology manifolds, and so they hold for homology manifolds as well by Remark \ref{rmk:BorelMoore_vs_Singular}.  The final step in Thurston's proof is an  Alexander duality result for homology spheres; the correct reference for this is \cite[Theorem VIII.6.4]{Wilder1965}.
}
\end{obs}
 
Thurston's result is particularly useful in conjunction with Vietoris-Begle's Theorem.
\begin{theo}[Vietoris-Begle's Mapping Theorem, \cite{spanier}, Chpt.6, Theorem 9.15]
	\label{theo:Vietoris-Begle} ${}$\\
	Let $\X$ and $\Y$ be two compact metric spaces, and let $f\colon \X \to \Y$ be surjective and continuous. Suppose that the fibers 
	$f^{-1}(y)$ are acyclic for every $y\in \Y$. Then, the induced homomorphism $f^*\colon H^*(\Y;\Z) \to H^*(\X;\Z)$ is an isomorphism.
\end{theo}

The first consequence is a generalization of \cite[Theorem 3d.1]{DavisJanuszkiewicz1991} for homology manifolds (not necessarily PL).

\begin{theo}
\label{theo:hom_mfd_group_infinity}
Let $(\X,\sfd)$ be a proper, $\CAT(0)$ homology $n$-manifold. Then:
    \begin{itemize}
        \item[(i)] the metric sphere $S_r(x)$ is a homology $(n-1)$-manifold which is a homology $(n-1)$-sphere for every $r>0$;
        \item[(ii)] if $X$  admits the cone at infinity, then for every $x\in \X$ there exists $r_x>0$ such that   the contraction maps $\varphi_x^{R,r} \colon S_R(x) \to S_r(x)$  induce surjective maps between fundamental groups, for  every $r  \in (0,r_x)$ and every $R\ge r$. 
    \end{itemize}
\end{theo}
\begin{proof}
    Let us call $f=\varphi_x^{R,r}: S_R(x) \to S_r(x)$ the contraction, for simplicity.   \\
    Applying Vietoris-Begle Mapping Theorem   \ref{theo:Vietoris-Begle}   to $f$, and using Proposition \ref{prop:Thurston},  we deduce that $f^\ast: H^k(S_R(x)) \to H^k(S_r(x))$ is an isomorphism, for every $0<r<R$ and every $k$. Then,  the two metric spheres $S_r(x)$ and $S_R(x)$ are homology $(n-1)$-manifolds, by Proposition \ref{prop:Thurston}, with the same integral cohomology groups, hence same homology,
    by Poincar\'e duality. Moreover, Propositions \ref{prop:characterization_hom_manifolds} and \ref{prop:homotopy_type_small_metric_spheres} imply that $H_\ast(S_r(x)) \simeq H_\ast(\mathbb{S}^{n-1})$ for every $r$ small enough, so $H_\ast(S_R(x)) \simeq H_\ast(\mathbb{S}^{n-1})$ for every $x\in \X$ and every $R>0$. This shows (i).
    
    To show (ii),
     notice that by the universal coefficient theorem we have  a commutative diagram (see \cite[Chapter V] {bredon-book})
$$\xymatrix{
   0  \ar[r]^{} 
&  \text{Ext} \big(H^{n} (S_R(x)), \Z \big)  
   \ar[r]^{}   
    \ar@{^{}->}[d]_{}
&  H_{n-1} (S_R(x))   
   \ar[r]^{} 
   \ar@{^{}->}[d]_{\Large f_\ast}
&  \text{Hom}\big( H^{n-1} (S_R(x)), \Z \big)
   \ar[r]^{}
    \ar@{^{}->}[d]_{}
&  0
\\
   0  \ar[r]^{} 
&   \text{Ext} \big(H^{n} (S_r(x)), \Z \big)
   \ar[r]^{}   
&  H_{n-1} (S_r(x))   
   \ar[r]^{} 
&  \text{Hom} \big( H^{n-1} (S_r(x)), \Z \big)
   \ar[r]^{}
&  0 }
$$
    and since $S_r(x), S_R(x)$ are $(n-1)$-homology spheres, the Ext groups are trivial and $f_\ast$ is an isomorphism too; therefore   $\vert \textup{deg}(f) \vert=1$.
    Now recall that, by Proposition \ref{prop:spheres_loc_contr_and_loc_path_connected}, for every $x$ there  exist $r_x,R_x>0$ such that all metric spheres centred at $x$ with radii in $(0,r_x)\cup (R_x,\infty)$ are locally path connected and locally contractible, so we can apply covering theory to these spheres.  
    Then, choose first  $R_0 \ge R_x$ and $r \le r_x < R_0$, let  $p:\widehat S_r(x) \to S_r(x)$ be the   covering associated to the subgroup $f_\ast \pi_1 (S_{R_0}(x))$, and  let  $\hat f:  S_{R_0}(x) \to \widehat S_r(x)$ be a lift of the contraction map. 
     The space $\widehat S_r(x)$ is  a  homology manifold too (since  it covers a   homology manifold) which is still compact and connected,   since  $\hat f$ is surjective by covering theory; moreover, $\widehat S_r(x)$ is  still a homology sphere, again by   Vietoris-Begle's theorem applied  to $\hat f$, which has the same fibers as  $f$; so it is also orientable.
    Since $1=\vert \textup{deg}(f) \vert = \vert \textup{deg}(\hat f \circ p) \vert = \vert \textup{deg}(\hat f) \cdot \textup{deg} (p) \vert$ by Lemma \ref{lemma:degree_is_multiplicative}, we deduce that      
 the degree of the (finite) covering $p$ is $\pm 1$. This in turns implies that  $f_\ast: \pi_1 (S_{R_0}(x)) \to \pi_1(S_r(x))$ is surjective. \\
 Finally, for any $r \le R < R_0$ we know that $f_\ast = (\varphi^{R,r}_x )_\ast \circ (\varphi^{R_0,R}_x)_\ast$ is surjective, hence $(\varphi^{R,r}_x )_\ast$ is surjective too.
\end{proof}

\begin{obs}{\em 
    It is likely that item (ii) of Theorem \ref{theo:hom_mfd_group_infinity} holds without assuming that $\X$  admits the cone at infinity. In order to use our strategy one would need to prove that each metric sphere $S_r(x)$ is locally contractible. While this is always true for small enough spheres,  and also for large enough spheres when $\X$ has the   cone at infinity, by Proposition \ref{prop:spheres_loc_contr_and_loc_path_connected}. \\ Alternatively one could try to follow a purely algebraic topology proof, but this goes beyond the scope of this paper.
    }
\end{obs}

A corollary is the following generalization of \cite[Theorem 4.7]{Nagano2022} to homology manifolds.
\begin{cor}
\label{cor:hom_mfds_homeo_to_R^n}
    Let $n\ge 3$. Let $(\X,\sfd)$ be a proper, $\CAT(0)$ homology $n$-manifold that admits the cone at infinity. If $\pi_1(\partial_T\X) = \{0\}$ then $\X$ is homeomorphic to $\R^n$.
\end{cor}
\begin{proof}
    A combination of Proposition \ref{prop:homotopy_stability_large_spheres} and item (ii) of Theorem \ref{theo:hom_mfd_group_infinity} implies that for every $x\in \X$ and $r>0$ small enough, the metric sphere $S_r(x)$ is simply connected. Therefore, it is homotopy equivalent to $\mathbb{S}^{n-1}$, being a homology $(n-1)$-manifold with the same homology of $\mathbb{S}^{n-1}$, by item (i) of Theorem \ref{theo:hom_mfd_group_infinity}. By Propositions \ref{prop:characterization_top_manifolds} and \ref{prop:homotopy_type_small_metric_spheres}, $\X$ is indeed a topological manifold. The thesis follows by \cite[Theorem 4.7]{Nagano2022}.
\end{proof}

\vspace{2mm}
\begin{proof}[Proof of Theorem \ref{theo:intro_homology_qi_R^n}]
    By Proposition \ref{prop:QI_to_R^n_has_biLip_asymptotic_cones}, $\X$ admits the cone at infinity $C_\infty \X$, which is moreover bi-Lipschitz equivalent to $\R^n$. Therefore, $\dim(C_\infty \X) = \dim(\R^n) = n$. Furthermore, by Proposition \ref{prop:properties_asymptotic_cone}, $\dim(C_\infty\X) = \dim(\X) = m$, hence $n=m$. If $n\ge 3$, Proposition \ref{prop:QI_to_R^n_has_biLip_asymptotic_cones} also implies that $\partial_T\X$ is homotopy equivalent to $\mathbb{S}^{n-1}$, hence simply connected. Therefore, Corollary \ref{cor:hom_mfds_homeo_to_R^n} implies that $\X$ is homeomorphic to $\R^n$. If $n=2$, then $\X$ is already a topological manifold and every contractible topological $2$-manifold is homeomorphic to $\R^2$.
\end{proof}

We end this section with a comment on Theorem \ref{theo:hom_mfd_group_infinity}. If $(\X,\sfd)$ is a proper, $\CAT(0)$ homology $n$-manifold then for every $x\in \X$ and every $0<r<R$, the contraction map $c_r^R\colon S_R(x) \to S_r(x)$ induces a map in homology which is an isomorphism. In general, this map is not a homotopy equivalence, actually the map induced at the level of fundamental groups is not an isomorphism in general. By point (ii) of Theorem \ref{theo:hom_mfd_group_infinity} this means that the map induced at the level of fundamental groups is not injective. 
There are examples of topological manifolds where this happens: in \cite[Theorem 5b.1]{DavisJanuszkiewicz1991} it is presented a proper, $\CAT(0)$ topological manifold which is not simply connected at infinity. This is precisely due to the fact that the contraction map between two metric spheres induces a map at the level of fundamental groups which is not injective.

\vspace{2mm}
\section{The proof of Theorems \ref{theo:intro_2}, \ref{theo:intro_2_volume} and \ref{theo:intro_sphere_degree}}
\label{sec-DE}

We start with Theorem \ref{theo:intro_sphere_degree}, which is the key to proving the Theorems \ref{theo:intro_2} and \ref{theo:intro_2_volume}.\\
It might be helpful,  for the reader, to visualize first the idea of proof of   Theorems   \ref{theo:intro_sphere_degree} and \ref{theo:intro_2_volume} in dimension $n=2$, which we give separately in the Appendix. 


\begin{T5}
Let $f\colon (\Y,\sfd) \to (\Y',\sfd')$ be a $1$-Lipschitz, surjective map between compact, geodesically complete, $\CAT(1)$ spaces of dimension $n$.
    Assume that $\Y$ is a homology manifold homotopy equivalent to $\mathbb{S}^n$, with $\mathcal{H}^n(\Y) < 2\cdot\mathcal{H}^{n}(\mathbb{S}^n)$: then,  $\Y'$ is also a homology $n$-manifold which is homotopy equivalent to $\mathbb{S}^n$.
\end{T5}

\begin{proof}
    For every $\Y'$ as in the statement, we define
    $$\textup{sph}(\Y'):= \max\{m\in\N\,:\, \Y' \cong \mathbb{S}^m \ast {\rm Z}\},$$
    (possibly, with ${\rm Z} = \emptyset$); in other words, $\textup{sph}(\Y')$ denotes the maximal round spherical factor of $\Y'$ in the unique decomposition of $\Y'$ in spherical joins provided by \cite[Corollary 1.2]{Lytchak2005}. Notice that the factor ${\rm Z}$ is a again a compact, geodesically complete $\CAT(1)$ space,   by \cite[Corollary II.3.15]{BH09}.   We will prove the statement by induction on $k:=n-\textup{sph}(\Y')$. \\
    If $k = 0$ then $\Y'$ is isometric to $\mathbb{S}^n$ and there is nothing to prove. \\
       If $k=1$ then $\Y' = \mathbb{S}^{n-1}\ast F$, where $F$ is a finite set of points since $\dim(\Y') = n$, and  $\sfd(p,q)=\pi$ for different $p,q \in F$. 
    Now, $\mathbb{S}^{n-1}\ast F$ consists of $\#F$ isometric copies of a spherical hemisphere, so  
    $$\#F\cdot \mathcal{H}^n(\mathbb{S}^n)/2=\mathcal{H}^n(\Y') \le \mathcal{H}^n(\Y) < 2\cdot \mathcal{H}^n(\mathbb{S}^n) .$$ Therefore  $\#F \le 3$. Notice that  $\#F \neq 1$, otherwise   $\Y$ would not be geodesically complete.  Moreover, $\#F \neq 2$ or   $\Y' \cong \mathbb{S}^n$, contradicting the definition of $\textup{sph}(\Y')$. \\
    Let us also exclude   that $\#F =3$, which will imply that $k \ge 2$ necessarily. Actually, if  $\#F = 3$ then  the space $\mathbb{S}^{n-1} \ast F$ would be isometric to the gluing of three spherical  hemispheres of  $\mathbb{S}^{n}$  along  their common boundary. Denoting by $H_i$  these   hemispheres, for $i=1,2,3$,    let us consider the three spheres $S_{ij} := H_i\cup H_j$ for every two distinct $i,j\in \{1,2,3\}$, and    
consider  the $1$-Lipschitz map $\rho\colon \Y'\to S_{12}$,
which is  the identity on $S_{12}$ and which coincides with the reflection    along the equator of $S_{23}$   on the hemisphere $H_3$.  
 Then,  $\hat{f} = \rho\circ f \colon \Y \to  S_{12}   \cong  \mathbb{S}^n$ is a $1$-Lipschitz surjective map between compact, connected, oriented, $\CAT(1)$ homology $n$-manifolds.  Corollary \ref{cor:degree_volume_inequality}  implies that
    $$2\cdot \mathcal{H}^n(\mathbb{S}^n) > \mathcal{H}^n(\Y) \ge \vert \textup{deg}(\hat{f})\vert \cdot \mathcal{H}^n(\mathbb{S}^n)$$
    therefore $\vert \textup{deg}(\hat{f})\vert \le 1$. 
    Actually, $\vert \textup{deg}(\hat{f})\vert = 1$ necessarily,  or $\textup{deg}(\hat{f}) =\textup{deg}_2(\hat{f})  = 0$ and then we would have  $\#\hat{f}^{-1}(y') \ge 2$ for  $\mathcal{H}^n$-a.e.  $y' \in S_{12}$ by Proposition \ref{prop:degree_area_formula_CAT}.(i)  (since   by construction $\hat f$ has at least two preimages for every $y'$ in the interior of $H_2$, which has   positive $\mathcal{H}^n$-measure). 
    By the area formula   \ref{prop:area_formula}  it would then follow
    $$2\cdot\mathcal{H}^n(\mathbb{S}^n)>\mathcal{H}^n(\Y)\ge \int_\Y \mathbf{J}f(y)\,\d\mathcal{H}^n(y) = \int_{S_{12}} \# \hat f^{-1}(y')\,\d\mathcal{H}^{n}(y') \ge 2\cdot\mathcal{H}^n(S_{12}) = 2\cdot\mathcal{H}^n(\mathbb{S}^n),$$
     which is a contradiction.
 Proposition \ref{prop:degree_area_formula_CAT}(i) then would imply that $\#\hat{f}^{-1}(y')\ge 3$ for every $y'\in H_2 \subseteq S_{12}$, so  applying again Proposition \ref{prop:area_formula} we would obtain
    $$2\cdot\mathcal{H}^n(\mathbb{S}^n)>\mathcal{H}^n(\Y)
    \ge  \int_{S_{12}} \# \hat f^{-1}(y')\,\d\mathcal{H}^{n}(y') \ge
    \mathcal{H}^{n}(H_1) + 3\cdot\mathcal{H}^n(H_2) = 2\cdot\mathcal{H}^n(\mathbb{S}^n),$$
    which is again a contradiction.
    
    Therefore,  $k\ge 2$ and now  we perform the induction. We   write $\Y' \cong \mathbb{S}^m \ast {\rm Z}$, where $m:=n-k$. 
    We choose any $z\in {\rm Z}$ and we see it as a point of $\Y'$. Lemma \ref{lemma:space_of_directions_to_spherical_join} tells us that $\Sigma_{z}\Y'\cong \mathbb{S}^m\ast \Sigma_z {\rm Z}$. Let $g \colon \Y' \to \mathbb{S}^0\ast \Sigma_{y'}\Y' \cong \mathbb{S}^{m+1}\ast\Sigma_z{\rm Z}$ be the $1$-Lipschitz surjective map given by Corollary \ref{cor:map_Y_to_S^0_join_Sigma_y}. We consider the $1$-Lipschitz surjective map $g\circ f \colon \Y \to \mathbb{S}^{m+1}\ast\Sigma_z{\rm Z}$. By the inductive hypothesis we deduce that $\mathbb{S}^{m+1}\ast\Sigma_z{\rm Z}$ is a homology $n$-manifold which is homotopy equivalent to $\mathbb{S}^n$; so, applying  Kunneth's formula for the join of two spaces (cp. \cite{milnor-kunneth}, Lemma 2.1)  we obtain that $\Sigma_z{\rm Z}$ is a homology $(k-2)$-manifold with the same homology of $\mathbb{S}^{k-2}$. By Proposition \ref{prop:characterization_hom_manifolds}, this implies that ${\rm Z}$ is a homology $(k-1)$-manifold. 
    Since $\Sigma_{z}\Y'\cong \mathbb{S}^m\ast \Sigma_z {\rm Z}$, again by Kunneth's formula and Proposition \ref{prop:characterization_hom_manifolds} we conclude that $\Y'$ is a homology $n$-manifold. \\ The next step is to show that $\Y'$ is orientable. If it is not orientable then $\textup{deg}_2(f) = 0$ by Lemma \ref{lemma:map_orientable_to_non_orientable_0_degree}. By Proposition \ref{prop:degree_area_formula_CAT} we have that $\#f^{-1}(y')$ is even for $\mathcal{H}^n$-a.e. $y'\in \Y'$. Since $f$ is surjective, we have $\#f^{-1}(y') \ge 2$ for $\mathcal{H}^n$-a.e. $y'\in \Y'$. We now apply Proposition \ref{prop:area_formula} to the $1$-Lipschitz map $f\colon \Y \to \Y'$ obtaining
     \begin{equation}\label{inequalityLN-N}
     2\cdot\mathcal{H}^n(\mathbb{S}^n)>\mathcal{H}^n(\Y) \ge \int_\Y \mathbf{J}f(y)\,\d\mathcal{H}^n = \int_{\Y'} \#f^{-1}(y')\,\d\mathcal{H}^{n} \ge 2\cdot\mathcal{H}^n(\Y) \ge 2\cdot\mathcal{H}^n(\mathbb{S}^n).
     \end{equation}
    where the last inequality follows by \cite[Proposition 11.3]{LN19} (originally proved  in \cite{Na02}). This is again a contradiction, which implies that $\Y'$ is orientable. Again, the last inequality in \eqref{inequalityLN-N} above, combined with Corollary \ref{cor:degree_volume_inequality}, yields
    $$2\cdot\mathcal{H}^n(\mathbb{S}^n) > \mathcal{H}^n(\Y) \ge \vert \textup{deg}(f)\vert\cdot \mathcal{H}^n(\Y') \ge 2\cdot \mathcal{H}^n(\mathbb{S}^n)$$
    which implies that $\vert \textup{deg}(f)\vert = 1$. 
    Then, a standard application of Poincar\'e duality, which holds for homology manifolds as explained in Section \ref{sec:orientability}, shows that  the homomorphism $f_*\colon H_*(\Y)\to H_*(\Y')$ is surjective in every degree:  in fact, for every $a \in H_k(Y')$ we have $a= \alpha \cap \xi_X  $, for some $\alpha \in H^{n-k}(Y')$ (notice that these are the usual singular homology groups and the usual cap product, since our spaces  are compact and locally contractible), so by naturality $f_\ast (f^\ast \alpha \cap \xi_Y)= \alpha \cap f_\ast \xi_Y = a$.
     Therefore, using that $H_\ast(\Y)\simeq H_\ast(\mathbb{S}^n)$ and that $H_n(\Y') = \Z$ (since $\Y'$ is orientable) we obtain that $H_\ast(\Y') \simeq H_\ast(\mathbb{S}^n)$. Now, as explained before Proposition \ref{prop:CAT(0)_hom_mfds_to_top_mfds_simply_connected}, $\Y'$ is homotopy equivalent to $\mathbb{S}^{n}$ if and only if it is simply connected. To show this, consider the universal cover $\pi\colon  \widetilde{\Y}'   \to \Y'$ and lift $f$ to  $\tilde{f}\colon \Y\to \widetilde{\Y}'$.
    By Lemma \ref{lemma:degree_is_multiplicative} we  get that $1 =\vert \textup{deg}(f) \vert = \vert \textup{deg}(\pi)\cdot \textup{deg}(\tilde{f}) \vert$, hence  $\textup{deg}(\pi) = \pm 1$. It follows that   $\pi$ is a homeomorphism, by Remark \ref{rmk:degree_covering}, so $Y'$ is simply connected and homotopically equivalent to $\mathbb S^n$.
\end{proof}

\noindent We can now prove   Theorems \ref{theo:intro_2} and \ref{theo:intro_2_volume}, which we restate below for the reader’s convenience.

\begin{T6}
    Let $(\X,\sfd)$ be a proper, geodesically complete, $\CAT(0)$ space which is quasi-isometric to $\R^n$. If $\mathcal{H}^{n-1}(\partial_T\X) <2 \cdot \mathcal{H}^{n-1}(\mathbb{S}^{n-1})$ then $\X$ is homeomorphic to $\R^n$.
\end{T6}

\begin{proof}
    By Proposition \ref{prop:QI_to_R^n_has_biLip_asymptotic_cones} we have that $\partial_T\X$ is a homology $(n-1)$-manifold which is homotopy equivalent to $\mathbb{S}^{n-1}$ and satisfies $\mathcal{H}^{n-1}(\partial_T\X) < 2\cdot\mathcal{H}^{n-1}(\mathbb{S}^{n-1})$ by assumption. 
    Proposition \ref{prop:properties_asymptotic_cone} implies that $\dim(\X) = n$. Let $x\in \X$ be an arbitrary point and let $\partial \log_x\colon \partial_T\X \to \Sigma_x\X$ be the $1$-Lipschitz surjective map defined in \eqref{eq:defin_log_from_Tits_boundary}. By applying Theorem \ref{theo:intro_sphere_degree}, we deduce that $\Sigma_x\X$ is 
    homotopy equivalent to $\mathbb{S}^{n-1}$, provided that $\dim(x) = n$. Therefore every  $x\in \X$ with $\dim(x) = n$ has the property that $T_x\X$ is purely $n$-dimensional, by Lemma \ref{lemma:CAT_homology_are_purely_dimensional}. 
    The proof of \cite[Proposition 8.1]{LN-finale-18} shows that then $\X$ is purely $n$-dimensional. Therefore, for every $x\in \X$, the space of directions $\Sigma_x\X$ is 
    homotopy equivalent to $\mathbb{S}^{n-1}$. In particular, $\X$ is a homology manifold by Proposition \ref{prop:characterization_hom_manifolds} (actually, a topological manifold, by Proposition \ref{prop:characterization_top_manifolds}). 
Therefore, $\X$ is homeomorphic to $\R^n$ by Theorem \ref{theo:intro_homology_qi_R^n}.
\end{proof}

\begin{T7}
    If $(\X,\sfd)$ is a proper, geodesically complete, $\CAT(0)$ space which is $(L,C)$-quasi-isometric to $\R^n$ for $L<2^{1/n}$, then $\X$ is homeomorphic to $\R^n$.
\end{T7}

\begin{proof}
    Arguing as in the  proof of Theorem \ref{theo:intro_qi_implies_biLip} we obtain that $\textup{as-vol}_n(\X) < 2$. Recalling \eqref{eqasvol-nagano} we have $\mathcal{H}^{n-1}(\partial_T\X) < 2\cdot \mathcal{H}^{n-1}(\mathbb{S}^{n-1})$. The conclusion then follows from Theorem \ref{theo:intro_2_volume}.
\end{proof}

\vspace{2mm}
\section{The proof of Theorem \ref{theo:intro_negative_answer} and Corollary \ref{cor:intro_hom_mfd_not_open}}
\label{sec:optimality_constant_2}

Until now, we gave positive answers to Question \ref{question:intro} under different assumptions. Now, we provide a negative one.

\begin{T3}
         For any $k\in \N^\ast$ there exists a proper, geodesically complete, purely $n$-dimensional, $\CAT(0)$  simplicial complex $(\X_k,\sfd_k)$ which is $(2,\frac{1}{k})$-quasi-isometric to $\R^n$, but   is not a homology manifold. Moreover, $\mathcal{H}^{n-1}(\partial_T\X_k) = 2\cdot\mathcal{H}^{n-1}(\mathbb{S}^{n-1})$ for every $k$.
\end{T3}

\begin{proof}
    First, we give an example in dimension $n=2$ and then we generalize   to every dimension.
    
\noindent  Let   $(S,\sfd)$  be a circle of total length $4\pi$, with distance   truncated to $\pi$;
      it is a compact, geodesically complete, $\CAT(1)$ space with diameter $\pi$. Let $\Y = \textup{Cone}(S)$ be the Euclidean cone over $S$, which is a proper, geodesically complete, $\CAT(0)$ space. 
    Observe that $S$ is $2$-biLipschitz homeomorphic to the standard circle $\mathbb{S}^1$, so $\Y$ is $2$-biLipschitz homeomorphic to $\R^2$. \\
      Now, choose   two points $v_1,v_2 \in S$ at    distance $2\pi$, with respect to the initial length metric, and consider the geodesic segments $\gamma_1(t) = tv_1$, $\gamma_2(t) = tv_2$ of $\Y$, for $t\in [0,\frac{1}{k}]$.
      We define $\X_k:=\Y/\!\sim$, by identifying     $\gamma_1 (t)\sim \gamma_2 (t)$   for   $t\in [0,\frac{1}{k}]$, and  equip $\X_k$ with the quotient metric. \\
      Geometrically, $\X_k$ is just the  degree $2$ ramified   covering of   $\mathbb R^2$ at the origin $O$,  locally isometric outside $O$, where we identified two geodesic segments of length $1/k$  from $O$, making  angle  $2\pi$.\\
To show that   $\X_k$   is a proper, $\CAT(0)$, geodesically complete space, it is useful to see   $\Y$ as the gluing of two (non-geodesically complete) $\CAT(0)$ spaces. 
Namely, we look at $S$ as the union of two half-circles of length $2\pi$, so the  Euclidean cones over each of these two half-circles are proper, $\CAT(0)$ spaces  with boundary $\Y_1$ and $\Y_2$.  Notice that each boundary $\partial \Y_i$ is a geodesic line, and  that  $\Y$ is obtained by gluing $\Y_1$ and  $\Y_2$ along these  boundary geodesic lines. \\
We may moreover assume that the segments $\gamma_1$ and  $\gamma_2$ make  angle $\pi$ with these geodesic lines. Therefore, the resulting space $\Y$  is $\CAT(0)$, because it is the gluing of two $\CAT(0)$ spaces along two convex subsets  
(the union of $\gamma_1$  with the boundary geodesic line $\partial \Y_1$   is a convex subset of $\Y_1$, since the angle between them  is exactly $\pi$; and the same for  $\gamma_2 \cup \partial \Y_2$),
see \cite[Theorem II.11.1]{BH09}.
 It is also clearly geodesically complete and proper. \\
The quotient map $\pi \colon \Y \to \X_k$ is a $(1,\frac{2}{k})$-quasi-isometry; therefore, $\X_k$ is $(2,\frac{2}{k})$-quasi-isometric to $\R^2$. Moreover, by construction we have $\mathcal{H}^{1}(\partial_T\X_k) = \mathcal{H}^{1}(S) = 2\cdot \mathcal{H}^{1}(\mathbb{S}^1)$. \\
  Notice that $\X_k$ is not homeomorphic to $\R^2$, since it has non-manifold points (all the points of the geodesic segment obtained by identifying $\gamma_1$ with $\gamma_2$).\\
  We also remark that $\X_k$ admits a structure of simplicial complex, and the metric is piecewise Euclidean on each simplex.  \\
    For $n\ge 3$, we just consider $\X_k \times \R^{n-2}$.
\end{proof}

 Figure \ref{fig:construction_counterexample} shows the construction of the space $\X_k$ in dimension $n=2$, while Figure \ref{fig:metric_spheres} highlights the structure of some metric spheres in   $\X_k$. In particular, the homology of large metric spheres is different from the homology of small ones.
 \\
    \begin{figure}[h!]

    \begin{tikzpicture}[scale=1.5]
    \def\radius{1}

    \draw[thick] (0,0) circle (\radius);

    \draw[dashed, black] (0, -\radius - 1) -- (0, \radius + 1) node[above] {};

    \node at (1.2*\radius, 0) {$2\pi$};

    \node at (1.2*\radius, 1.2*\radius) {$\Y_2$};

    \node at (-1.2*\radius, 0) {$2\pi$};

    \node at (-1.2*\radius, 1.2*\radius) {$\Y_1$};


    \draw[thick] (4,0) circle (\radius);

    \draw[dashed, black] (4, -\radius - 1) -- (4, \radius + 1) node[above] {};

    \node at (4+1.2*\radius, 0) {$2\pi$};

    \node at (4+1.2*\radius, 1.2*\radius) {$\Y_2$};

    \node at (4-1.2*\radius, 0) {$2\pi$};

    \node at (4 -1.2*\radius, 1.2*\radius) {$\Y_1$};

    \draw[thick, red] (4, 0) -- (3.5, 0) node[above] {};

    \node at (3.3,0) {$\gamma_1$};

    \node at (4.7,0) {$\gamma_{2}$};

    \draw[thick, green] (4, 0) -- (4.5, 0) node[above] {};

    \draw[] (4,0) circle (\radius/8);

    \node at (4-\radius/4, \radius/4) {$\pi$};

    \node at (4-\radius/4, -\radius/4) {$\pi$};

    \node at (4+\radius/4, \radius/4) {$\pi$};

    \node at (4+\radius/4, -\radius/4) {$\pi$};
\end{tikzpicture}

    \caption{On the left, the construction of $\Y$ as the gluing of the two $\CAT(0)$ spaces $\Y_1$ and $\Y_2$ along their common boundary. On the right, the construction of   $\X_k$: the geodesic segments $\gamma_1$ (in red) and $\gamma_{2}$ (in green) are identified.\\ The resulting space is still $\CAT(0)$ because  all the displayed angles have value $\pi$.}
    \label{fig:construction_counterexample}
\end{figure}
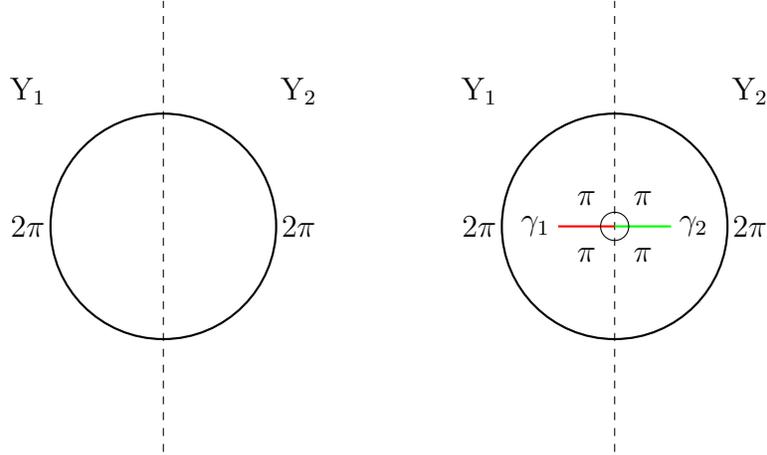
\begin{figure}[h!]
\begin{tikzpicture}[scale=1]
    
    \draw[] (0, -2.5) -- (0, 2.5) node[above] {};

    \draw[red, thick] (-1,0) -- (0,0) node[above] {};

    \draw[green, thick] (1,0) -- (0,0) node[above] {};

    \fill (-1, 0) circle (1.5pt) node[below left] {$x$};
    \fill (1, 0) circle (1.5pt) node[below right] {$x$};

    \draw[black] (-1, 0) ++(300:2) arc (300:60:2);

    \draw[black] (1, 0) ++(-120:2) arc (-120:120:2);

    \node at (1.3,1.3) {$S_R(x) \simeq \mathbb{S}^1$};


    \draw[] (8, -2.5) -- (8, 2.5) node[above] {};

    \draw[red, thick] (7,0) -- (8,0) node[above] {};

    \draw[green, thick] (9,0) -- (8,0) node[above] {};

    \fill (7, 0) circle (1.5pt) node[below left] {$x$};
    \fill (9, 0) circle (1.5pt) node[below right] {$x$};

    \node at (9.3,1) {$S_r(x) \not\simeq \mathbb{S}^1$};

    \draw[] (7,0) circle (0.5);

    \draw[] (9,0) circle (0.5);

    \fill (7.5, 0) circle (1.5pt) node[below right] {$y$};
    \fill (8.5, 0) circle (1.5pt) node[below left] {$y$};
    
\end{tikzpicture}

\caption{The figure shows two metric spheres around the point $x$. \\
When the radius is large (as shown on the left), the metric sphere around $x$ is homeomorphic to $\mathbb{S}^1$; when the radius is small (as  shown on the right), the metric sphere is homeomorphic to the gluing of the two circles by their common point $y$. \\
In particular, the metric sphere is not homeomorphic to $\mathbb{S}^1$.\\
The fiber of the point $y$ along the contraction map $c_r^R\colon S_R(x) \to S_r(x)$ consists of two points; in particular, it is not acyclic since it is not trivial in degree $0$.}
\label{fig:metric_spheres}
\end{figure}
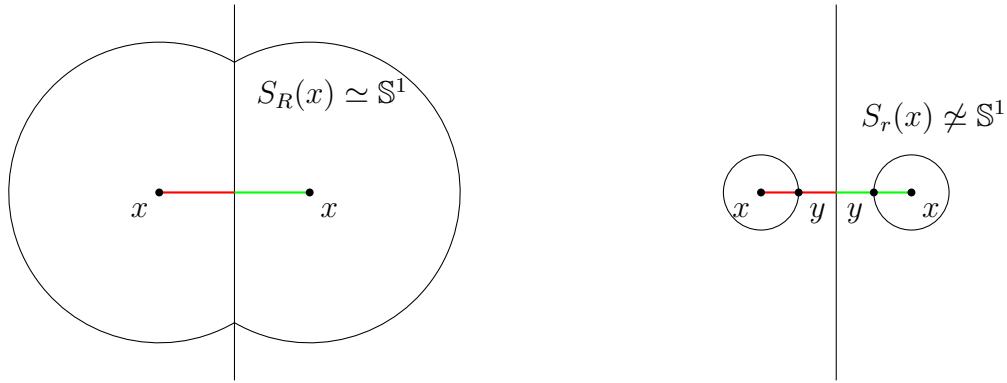

\begin{T4}
    The class of proper, $\CAT(0)$ homology (or topological) manifolds is not an open subset of the class of all proper, geodesically complete  $\CAT(0)$ spaces, with respect to the Gromov-Hausdorff topology.
\end{T4}

\begin{proof}
    Let $(\X_k,\sfd_k)$ be the sequence considered   in the proof of Theorem \ref{theo:intro_negative_answer}, for $k \to \infty$. \\Then, one of its pointed limits is the space $\Y$, which is biLipschitz homeomorphic to $\R^2$, hence a manifold; but each $\X_k$ is not a homology manifold.
\end{proof}

\vspace{2mm}
\section*{Appendix: the proofs of Section \ref{sec-DE} in dimension $2$}

In this appendix we give a sketch of a more intuitive version of the proofs of Theorems  \ref{theo:intro_sphere_degree} and \ref{theo:intro_2_volume}  of Section \ref{sec-DE}, in dimension $n=2$. In order to do so, we start by describing the possible local geometry of $2$-dimensional geodesically complete $\CAT(0)$ spaces that are $(L,C)$-quasi-isometric to $\R^2$ with $L<2$. For this reason we introduce the following space.

\begin{defin}
    Let $a\ge b \ge c >0$ be three positive real numbers. We denote by $G(a,b,c)$ the graph obtained by taking two vertices $v,w$ and three edges connecting $v$ to $w$ of length respectively $a,b$ and $c$. The three edges are denoted respectively by $e_a,e_b,e_c$; let moreover $C_{ac}$ be the loop composed by the concatenation of the edges $e_a$ and $e_c$, and  $C_{bc}$ the one composed by  the edges   $e_b$ and $e_c$. Figure \ref{fig:graph_a_b_c} shows a representation of this graph.
\end{defin}

\begin{figure}[h!]

\begin{tikzpicture}[scale = 1.8]

\def\rRight{0.8} 
\def\rLeft{1.2}  

\draw[black, thick] (0,0) coordinate (B) -- (0,1) coordinate (A);

\pgfmathsetmacro{\cXRight}{sqrt(\rRight^2 - 0.5^2)}
\coordinate (C_R) at (\cXRight, 0.5);

\pgfmathsetmacro{\angleAR}{atan2(1-0.5, 0-\cXRight)} 
\pgfmathsetmacro{\angleBR}{atan2(0-0.5, 0-\cXRight)} 

\draw[blue, thick] (A) arc (\angleAR:\angleBR:\rRight);

\pgfmathsetmacro{\cXLeft}{sqrt(\rLeft^2 - 0.5^2)}
\coordinate (C_L) at (-\cXLeft, 0.5);

\pgfmathsetmacro{\angleBL}{atan2(0-0.5, 0-(\cXLeft))} 
\pgfmathsetmacro{\angleAL}{atan2(1-0.5, 0-(\cXLeft))} 

\draw[red, thick] (B) arc (180 - \angleBL:180 -\angleAL:\rLeft);

\node[right] at (1.4,0.5) {$b$};
\node[left] at (-2,0.5) {$a$};
\node[left] at (0,0.5) {$c$};

\fill (0, 0) circle (1.5pt) node[right] {$w$};
\fill (0, 1) circle (1.5pt) node[left] {$v$};

\end{tikzpicture}

\caption{The figure shows the graph $G(a,b,c)$ with two vertices $v,w$ and three edges of length $a$ (in red), $b$ (in blue) and $c$ (in black).}
\label{fig:graph_a_b_c}
\end{figure}
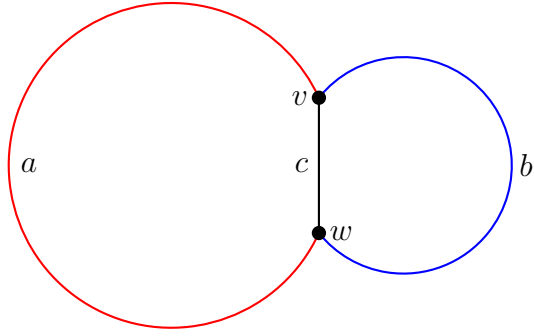

Given $\ell >0$, we denote by $\mathbb{S}^1_\ell$ the metric space $(\mathbb{S}^1, \sfd_\ell)$ where $\sfd_\ell$ is the only geodesic distance on $\mathbb{S}^1$ with total length $\ell$. The following proposition describes the structure of the possible space of directions $\Sigma$ of a $2$-dimensional $\CAT(0)$-space,  with $\mathcal{H}^{1} (\Sigma) < 2 \cdot\mathcal{H}^{1}(\mathbb{S}^{1})$. 

\begin{prop}
\label{prop:CAT(1)_graphs_small_volume}
    Let $(\Y,\sfd)$ be a compact, connected, geodesically complete, $\CAT(1)$ space with $\dim(\Y) = 1$. If $\mathcal{H}^1(\Y) < 4\pi$ then $\Y$ is isometric to either $\mathbb{S}^1_\ell$ with $\ell < 4\pi$ or to $G(a,b,c)$ with $a+b+c < 4\pi$ and $a+b+2c \ge 4\pi$.
\end{prop}
\begin{proof}[Sketch of the proof]
    $\Y$ is a connected, finite graph equipped with the path metric. Now, we have two constraints. On the one hand, by geodesic completeness, every vertex has degree at least $2$. On the other, by the $\CAT(1)$ condition, the length of every loop is at least $2\pi$. Using this two constraints it is not difficult to show that $\Y$ is isometric to either $\mathbb{S}^1_\ell$ for some $2\pi \le \ell < 4\pi$ (namely, when every vertex has degree $2$) or to $G(a,b,c)$. The condition $a+b+c < 4\pi$ follows from the bound on the total length, while the bound $a+b+2c \ge 4\pi$ comes from the fact that the two loops $C_{ac}$ and $C_{bc}$
    must have length at least $2\pi$.
\end{proof}

The next result corresponds to Theorem \ref{theo:intro_sphere_degree} in dimension $2$. The original proof was suggested by Noa Vikman. 
We just sketch it.
\begin{prop}
    \label{prop:maps_on_G(a,b,c)}
    Let $f\colon \mathbb{S}^1_\ell \to G(a,b,c)$ be a $1$-Lipschitz surjective map. Then $\ell \ge a+b+2c$.
\end{prop}
\begin{proof}[Sketch of the proof]
 There is a continuous map $g_a\colon G(a,b,c) \to \mathbb{S}^1_{a+c} \cong  C_{ac}  $  with $g(e_a)=e_a$ and one preimage for every point  of the edge $e_a$; and,  analogously,  a  map $g_b\colon G(a,b,c) \to  \mathbb{S}^1_{b+c} \cong  C_{bc}$ with $g(e_b)=e_b$ and one preimage  for every   $z \in e_b$.
    Suppose that one of the two maps $g_a\circ f: S^1_\ell \to \mathbb{S}^1_{a+c} $ and  $g_b\circ f \colon : S^1_\ell \to \mathbb{S}^1_{b+c} $ has degree zero, say $g_a\circ f$ (the other case being analogous). Then, every point of $\mathbb{S}^1_{a+c}$ has an even, and positive, number of preimages; in particular, for every $z \in g_a(e_a)$  we get that $\#f^{-1}(z)\ge 2$. Applying  Proposition \ref{prop:area_formula} we obtain that
    $$\ell = \mathcal{H}^1(\mathbb{S}^1_\ell) \ge 2a+b+c \ge a+b+2c.$$
    The other possible case, due to the area formula, is that both $g_a\circ f, g_b\circ f$ have degree one. Once again, this implies that the number of preimages on some edge is at least $2$, which yields as before $\ell \ge a+b+2c$ by the area formula.
\end{proof}

\begin{proof}[Proof of Theorem \ref{theo:intro_2_volume} in dimension $2$]
    Let $(\X,\sfd)$ be a proper, geodesically complete, $\CAT(0)$ space of dimension $2$ which is quasi-isometric to $\R^2$ and such that $\mathcal{H}^1(\partial_T\X) < 4\pi$. Remark \ref{rmk:2-dim-qi_implies_pure_dimensional} implies that $\X$ is purely $2$-dimensional. Propositions \ref{prop:QI_to_R^n_has_biLip_asymptotic_cones} and \ref{prop:CAT(1)_graphs_small_volume} imply that $\partial_T\X$ is isometric to $\mathbb{S}^1_\ell$ for some $2\pi\le\ell<4\pi$. For every $x\in \X$ we consider the $1$-Lipschitz, surjective map  $\partial\log_x\colon \partial_T\X \to \Sigma_x\X$ given by  \eqref{eq:defin_log_from_Tits_boundary}, Section \ref{sec:titsboundary}, from which we deduce that $\mathcal{H}^1(\Sigma_x\X) < 4\pi$. The space $\Sigma_x\X$ is a compact, geodesically complete, $\CAT(1)$ space of $\dim(\Sigma_x\X) = 1$ which is connected, being the continuous image of $\partial_T\X$. Proposition \ref{prop:CAT(1)_graphs_small_volume} implies that $\Sigma_x\X$ is isometric to either $\mathbb{S}^1_{\ell'}$, with $2\pi \le \ell' < 4\pi$ or to $G(a,b,c)$ with $a+b+c < 4\pi$ and $a+b+2c \ge 4\pi$. However, the second case cannot occur, otherwise, considering again the map $\partial\log_x \colon \mathbb{S}^1_\ell \to \Sigma_x\X \cong G(a,b,c)$, we would deduce   by Proposition \ref{prop:maps_on_G(a,b,c)}  that $\ell \ge a+b+2c \ge 4\pi$, a contradiction. Therefore, $\Sigma_x\X = \mathbb{S}^1_{\ell'}$ for some $2\pi\le \ell' < 4\pi$ for every $x\in \X$. In particular, $\X$ is a topological manifold by Proposition \ref{prop:characterization_top_manifolds}. Therefore, $\X$ is homeomorphic to $\R^2$ by Theorem \ref{theo:intro_homology_qi_R^n}.
\end{proof}

	\bibliographystyle{alpha}
	\bibliography{biblio}
	
\end{document}